\pdfoutput=1
\documentclass[11pt, a4paper, twoside,leqno]{amsart}
\usepackage[centering, totalwidth = 380pt, totalheight = 600pt]{geometry}
\usepackage{amssymb, amsmath, amsthm}
\usepackage{microtype, stmaryrd, url, lmodern, eucal, bbm,bm}
\usepackage[rm={lining=true}]{cfr-lm}
\usepackage[shortlabels]{enumitem}
\usepackage[latin1]{inputenc}
\usepackage{color}
\usepackage[usenames, dvipsnames, svgnames, table]{xcolor}
\usepackage[colorlinks=true, urlcolor=blue, linkcolor=blue, citecolor=Maroon]{hyperref}
\usepackage[british]{babel}

\makeatletter
\def\@cite#1#2{[{#1\if@tempswa ,~#2\fi}]}
\makeatother
\setlist{labelsep=5pt,leftmargin=*}

\usepackage[arrow, matrix, tips, curve, graph, rotate]{xy}
\SelectTips{cm}{10}

\makeatletter
\def\matrixobject@{%
  \edef \next@{={\DirectionfromtheDirection@ }}%
  \expandafter \toks@ \next@ \plainxy@
  \let\xy@@ix@=\xyq@@toksix@
  \xyFN@ \OBJECT@}
\let\xy@entry@@norm=\entry@@norm
\def\entry@@norm@patched{%
  \let\object@=\matrixobject@
  \xy@entry@@norm }
\AtBeginDocument{\let\entry@@norm\entry@@norm@patched}
\makeatother

\newcommand{\twocong}[2][0.5]{\ar@{}[#2] \save ?(#1)*{\cong}\restore}
\newcommand{\twoeq}[2][0.5]{\ar@{}[#2] \save ?(#1)*{=}\restore}
\newcommand{\rtwocell}[3][0.5]{\ar@{}[#2] \ar@{=>}?(#1)+/l 0.2cm/;?(#1)+/r 0.2cm/^{#3}}
\newcommand{\ltwocell}[3][0.5]{\ar@{}[#2] \ar@{=>}?(#1)+/r 0.2cm/;?(#1)+/l 0.2cm/^{#3}}
\newcommand{\ltwocello}[3][0.5]{\ar@{}[#2] \ar@{=>}?(#1)+/r 0.2cm/;?(#1)+/l 0.2cm/_{#3}}
\newcommand{\dtwocell}[3][0.5]{\ar@{}[#2] \ar@{=>}?(#1)+/u  0.2cm/;?(#1)+/d 0.2cm/^{#3}}
\newcommand{\dltwocell}[3][0.5]{\ar@{}[#2] \ar@{=>}?(#1)+/ur  0.2cm/;?(#1)+/dl 0.2cm/^{#3}}
\newcommand{\drtwocell}[3][0.5]{\ar@{}[#2] \ar@{=>}?(#1)+/ul  0.2cm/;?(#1)+/dr 0.2cm/^{#3}}
\newcommand{\dthreecell}[3][0.5]{\ar@{}[#2] \ar@3{->}?(#1)+/u  0.2cm/;?(#1)+/d 0.2cm/^{#3}}
\newcommand{\utwocell}[3][0.5]{\ar@{}[#2] \ar@{=>}?(#1)+/d 0.2cm/;?(#1)+/u 0.2cm/_{#3}}
\newcommand{\dtwocelltarg}[3][0.5]{\ar@{}#2 \ar@{=>}?(#1)+/u  0.2cm/;?(#1)+/d 0.2cm/^{#3}}
\newcommand{\utwocelltarg}[3][0.5]{\ar@{}#2 \ar@{=>}?(#1)+/d  0.2cm/;?(#1)+/u 0.2cm/_{#3}}
\newcommand{\cell}[3][0.5]{\ar@{}[#2]|(#1){#3}}

\newcommand{\pullbackcorner}[1][dr]{\save*!/#1-1.4pc/#1:(-1,1)@^{|-}\restore}

\newdir{(}{{}*!<0em,-.14em>-\cir<.14em>{l^r}}
\newdir{ (}{{}*!/-5pt/\dir{(}}
\newdir{ >}{{}*!/-5pt/\dir{>}}
\newcommand{\sh}[2]{**{!/#1 -#2/}}

\DeclareMathOperator{\el}{el}

\newcommand{\cat}[1]{\mathrm{\mathcal #1}}
\newcommand{\thg}{{\mathord{\text{--}}}}

\newcommand{\abs}[1]{{\left|{#1}\right|}}

\newcommand{\res}[2]{\left.{#1}\right|_{#2}}

\newcommand{\cd}[2][]{\vcenter{\hbox{\xymatrix#1{#2}}}}

\renewcommand{\phi}{\varphi}
\newcommand{\A}{{\mathcal A}}
\newcommand{\B}{{\mathcal B}}
\newcommand{\C}{{\mathcal C}}

\newcommand{\E}{{\mathcal E}}

\renewcommand{\O}{{\mathcal O}}

\let\sec=\S
\renewcommand{\S}{{\mathcal S}}
\newcommand{\T}{{\mathcal T}}

\newcommand{\xtor}[1]{\cdl[@1]{{} \ar[r]|-{\object@{|}}^{#1} & {}}}

\makeatletter

\def\hookleftarrowfill@{\arrowfill@\leftarrow\relbar{\relbar\joinrel\rhook}}
\def\twoheadleftarrowfill@{\arrowfill@\twoheadleftarrow\relbar\relbar}
\def\leftbararrowfill@{\arrowdoublefill@{\leftarrow\mkern-5mu}\relbar\mapstochar\relbar\relbar}
\def\Leftbararrowfill@{\arrowdoublefill@{\Leftarrow\mkern-2mu}\Relbar\Mapstochar\Relbar\Relbar}
\def\leftringarrowfill@{\arrowdoublefill@{\leftarrow\mkern-3mu}\relbar{\mkern-3mu\circ\mkern-2mu}\relbar\relbar}
\def\lefttriarrowfill@{\arrowfill@{\mathrel\triangleleft\mkern0.5mu\joinrel\relbar}\relbar\relbar}
\def\Lefttriarrowfill@{\arrowfill@{\mathrel\triangleleft\mkern1mu\joinrel\Relbar}\Relbar\Relbar}

\def\hookrightarrowfill@{\arrowfill@{\lhook\joinrel\relbar}\relbar\rightarrow}
\def\twoheadrightarrowfill@{\arrowfill@\relbar\relbar\twoheadrightarrow}
\def\rightbararrowfill@{\arrowdoublefill@{\relbar\mkern-0.5mu}\relbar\mapstochar\relbar\rightarrow}
\def\Rightbararrowfill@{\arrowdoublefill@{\Relbar\mkern-2mu}\Relbar\Mapstochar\Relbar\Rightarrow}
\def\rightringarrowfill@{\arrowdoublefill@\relbar\relbar{\mkern-2mu\circ\mkern-3mu}\relbar{\mkern-3mu\rightarrow}}
\def\righttriarrowfill@{\arrowfill@\relbar\relbar{\relbar\joinrel\mkern0.5mu\mathrel\triangleright}}
\def\Righttriarrowfill@{\arrowfill@\Relbar\Relbar{\Relbar\joinrel\mkern1mu\mathrel\triangleright}}

\def\leftrightarrowfill@{\arrowfill@\leftarrow\relbar\rightarrow}
\def\mapstofill@{\arrowfill@{\mapstochar\relbar}\relbar\rightarrow}

\renewcommand*\xleftarrow[2][]{\ext@arrow 20{20}0\leftarrowfill@{#1}{#2}}
\providecommand*\xLeftarrow[2][]{\ext@arrow 60{22}0{\Leftarrowfill@}{#1}{#2}}
\providecommand*\xhookleftarrow[2][]{\ext@arrow 10{20}0\hookleftarrowfill@{#1}{#2}}
\providecommand*\xtwoheadleftarrow[2][]{\ext@arrow 60{20}0\twoheadleftarrowfill@{#1}{#2}}
\providecommand*\xleftbararrow[2][]{\ext@arrow 10{22}0\leftbararrowfill@{#1}{#2}}
\providecommand*\xLeftbararrow[2][]{\ext@arrow 50{24}0\Leftbararrowfill@{#1}{#2}}
\providecommand*\xleftringarrow[2][]{\ext@arrow 10{26}0\leftringarrowfill@{#1}{#2}}
\providecommand*\xlefttriarrow[2][]{\ext@arrow 80{24}0\lefttriarrowfill@{#1}{#2}}
\providecommand*\xLefttriarrow[2][]{\ext@arrow 80{24}0\Lefttriarrowfill@{#1}{#2}}

\renewcommand*\xrightarrow[2][]{\ext@arrow 01{20}0\rightarrowfill@{#1}{#2}}
\providecommand*\xRightarrow[2][]{\ext@arrow 04{22}0{\Rightarrowfill@}{#1}{#2}}
\providecommand*\xhookrightarrow[2][]{\ext@arrow 00{20}0\hookrightarrowfill@{#1}{#2}}
\providecommand*\xtwoheadrightarrow[2][]{\ext@arrow 03{20}0\twoheadrightarrowfill@{#1}{#2}}
\providecommand*\xrightbararrow[2][]{\ext@arrow 01{22}0\rightbararrowfill@{#1}{#2}}
\providecommand*\xRightbararrow[2][]{\ext@arrow 04{24}0\Rightbararrowfill@{#1}{#2}}
\providecommand*\xrightringarrow[2][]{\ext@arrow 01{26}0\rightringarrowfill@{#1}{#2}}
\providecommand*\xrighttriarrow[2][]{\ext@arrow 07{24}0\righttriarrowfill@{#1}{#2}}
\providecommand*\xRighttriarrow[2][]{\ext@arrow 07{24}0\Righttriarrowfill@{#1}{#2}}

\providecommand*\xmapsto[2][]{\ext@arrow 01{20}0\mapstofill@{#1}{#2}}
\providecommand*\xleftrightarrow[2][]{\ext@arrow 10{22}0\leftrightarrowfill@{#1}{#2}}
\providecommand*\xLeftrightarrow[2][]{\ext@arrow 10{27}0{\Leftrightarrowfill@}{#1}{#2}}

\makeatother

\numberwithin{equation}{section}

\theoremstyle{plain}
\newtheorem*{Thm*}{Theorem}
\newtheorem{Thm}{Theorem}
\newtheorem{Prop}[Thm]{Proposition}
\newtheorem{Cor}[Thm]{Corollary}
\newtheorem{Lemma}[Thm]{Lemma}

\theoremstyle{definition}
\newtheorem{Defn}[Thm]{Definition}

\newtheorem{Ex}[Thm]{Example}

\newtheorem{Rk}[Thm]{Remark}

\newcommand{\Cat}{\cat{Cat}}
\newcommand{\Set}{\cat{Set}}

\newcommand{\atwo}{{\mathbbm 2}}

\newcommand{\fibre}{\phi}
\newcommand{\somesetvaluedfunctor}{Q}
\newcommand{\someothercat}{{\mathcal{C}'}}

\begin{document}
\title{Operadic categories and d\'ecalage}
\author{Richard Garner} 
\address{Centre of Australian Category Theory, Macquarie University, NSW 2109, Australia} 
\email{richard.garner@mq.edu.au}
\author{Joachim Kock}
\address{Departament de matem\`atiques, Universitat Aut\`onoma de Barcelona}
\email{kock@mat.uab.cat}
\author{Mark Weber}
\address{Faculty of Mathematics and Physics, Charles University, Prague}
\email{mark.weber.math@gmail.com}

\subjclass[2010]{Primary: 18D50, 18C15, 18C20}
\date{\today}

\thanks{The authors acknowledge, with gratitude, the following grant
  funding. Garner was supported by Australian Research Council grants
  DP160101519 and FT160100393; Kock by grants MTM2016-80439-P
  (AEI/FEDER, UE) of Spain and 2017-SGR-1725 of Catalonia; and Weber
  by Czech Science Foundation grant GA CR P201/12/G028.}

\begin{abstract}
  Batanin and Markl's operadic categories are categories in which each
  map is endowed with a finite collection of ``abstract
  fibres''---also objects of the same category---subject to suitable
  axioms. We give a reconstruction of the data and axioms of operadic
  categories in terms of the d\'ecalage comonad $\mathsf{D}$ on small
  categories. A simple case involves \emph{unary} operadic
  categories---ones wherein each map has exactly one abstract
  fibre---which are exhibited as categories which are, first of all,
  coalgebras for the comonad $\mathsf{D}$, and, furthermore, algebras
  for the monad $\mathsf{\tilde D}$ induced on $\cat{Cat}^\mathsf{D}$
  by the forgetful--cofree adjunction. A similar description is found
  for general operadic categories arising out of a corresponding
  analysis that starts from a ``modified d\'ecalage'' comonad
  $\mathsf{D}_m$ on the arrow category $\cat{Cat}^\atwo$.
\end{abstract}
\maketitle

\section{Introduction}

\emph{Operads} originated in algebraic topology, first appearing in
Boardman and Vogt~\cite{Boardman1968Homotopy-everything} under the name
``category of operators in standard form'', with the modern name and
modern definition being provided by May in~\cite{May1972The-geometry}.
They quickly caught on, with applications subsequently being found not
only in topology, but also in algebra, geometry, physics and beyond;
see~\cite{Markl2002Operads} for an overview. As the use of operads has
grown, it has proven useful to recast the definition: rather than
explicitly listing the data and axioms, one may re-express them in
various more abstract ways~\cite{Baez1998Higher-Dimensional,
  Joyal1986Foncteurs, Kelly2005On-the-operads}, each of which points
towards a range of practically useful generalisations of the original
notion.

This has led to a rich profusion of operad-like structures, and
various authors have proposed unifying frameworks to bring some order
to this proliferation. One such framework is that of \emph{operadic
  categories}~\cite{Batanin2015Operadic}, introduced by Batanin
and Markl to specify certain kinds of generalised operad necessary for
their proof of the duoidal Deligne conjecture. An operadic category is
a combinatorial object which specifies a flavour of operad; an
``algebra'' for an operadic category is an operad of that flavour.
Such an operad will, in turn, have its own algebras, but this extra
layer will not concern us here.

As the name suggests, operadic categories are categories, but endowed
with extra structure of a somewhat delicate nature. This structure
seems to invite attempts at reconfiguration, so as better to link it
to other parts of the mathematical landscape. One such reconfiguration
was given by Lack~\cite{Lack2018Operadic}, who drew a tight
correspondence between operadic categories and the \emph{skew-monoidal
  categories} of Szlach\' anyi~\cite{Szlachanyi2012Skew-monoidal},
which in recent years have figured prominently in categorical quantum
algebra and work of the Australian school of category theorists.

The present paper gives another reconfiguration of the definition of
operadic category, which links it to the (upper) \emph{d\' ecalage}
construction. While primarily an operation on simplicial sets,
d\'ecalage may also---via the nerve functor---be seen as an operation
on categories; namely, that which takes a category to the disjoint
union of its~slices:
\begin{equation*}
    D(\C) \ = \ \textstyle\sum_{X \in \C} \C / X\rlap{ .}
\end{equation*}
  
There are two main aspects to the tight relationship between operadic
categories and d\'ecalage. To explain these, we must first recall 
the data for an operadic category. These are: a small category $\C$ with a
chosen terminal object in each connected component; a cardinality
functor $\abs{\thg} \colon \C \rightarrow \S$ into the category of
finite ordinals and arbitrary mappings; and an operation assigning to
every $f \colon Y \rightarrow X$ in $\C$ and $i \in \abs X$ an
``abstract fibre'' $f^{-1}(i) \in \C$, functorially in $Y$.

The first connection between operadic categories and d\'ecalage arises
from the fact that the d\'ecalage construction on categories underlies
a \emph{comonad} $\mathsf{D}$ on $\cat{Cat}$, whose coalgebras may be
identified, as in Proposition~\ref{prop:D-CoAlg} below, with
categories endowed with a choice of terminal object in each connected
component. In particular, each operadic category is a coalgebra for
the d\'ecalage comonad.

The second connection arises through the functorial assignation of
abstract fibres $f \mapsto f^{-1}(i)$ in an operadic category.
Functoriality says that, for fixed $X \in \C$ and $i \in \abs X$, this
assignation is the action on objects of a functor
$\varphi_{X,i} \colon \C / X \rightarrow \C$, so that the totality of
the abstract fibres can be expressed via a single functor
\begin{equation}\label{eq:9}
  \varphi \colon \textstyle\sum_{X \in \C, i \in \abs X} \C / X \rightarrow \C\rlap{ .}
\end{equation}
The domain of this functor is clearly related to the d\'ecalage of
$\C$, and in due course, we will explain it in terms of a
\emph{modified d\'ecalage} construction on categories endowed with a
functor to $\S$. However, there is a special case where no
modification is necessary. We call an operadic category \emph{unary}
if each $\abs X$ is a singleton; in this case, the domain
of~\eqref{eq:9} is precisely the d\'ecalage $D(\C)$, so that the fibres
of a unary operadic category are encoded in a single functor
$D(\C) \rightarrow \C$.

So, for a unary operadic category $\C$, we have on the one hand, that
$\C$ is a $\mathsf D$-coalgebra; and on the other, that $\C$ is
endowed with a map $D(\C) \rightarrow \C$. To reconcile these
apparently distinct facts, we apply a general observation: any comonad
$\mathsf{C}$ on a category $\A$ induces a monad $\mathsf{\tilde C}$ on
the category of $\mathsf C$-coalgebras $\A^\mathsf C$, namely, the
monad generated by the forgetful--cofree adjunction
$\A^\mathsf{C} \leftrightarrows \A$. In the case of the d\'ecalage
comonad, we induce a d\'ecalage monad $\mathsf{\tilde D}$ on
$\cat{Cat}^\mathsf{D}$;
and the axioms of a unary operadic category turn out to be
captured \emph{precisely} by the requirement that the map
$D(\C) \rightarrow \C$ giving abstract fibres should endow the
$\mathsf D$-coalgebra $\C$ with $\mathsf{\tilde D}$-algebra structure
in $\cat{Cat}^\mathsf{D}$. Our first main result is thus:

\begin{Thm*}
  The category of algebras for the d\'ecalage monad
  $\mathsf{\tilde{D}}$ on $\cat{Cat}^\mathsf{D}$ is isomorphic to the
  category of unary operadic categories.
\end{Thm*}

In order to remove the qualifier ``unary'' from this theorem and
accommodate the ``multi'' aspect of the general definition, we will
need, as anticipated above, to adjust the d\'ecalage construction.
Rather than the d\'ecalage comonad $\mathsf{D}$ on $\cat{Cat}$, we
will consider a \emph{modified d\'ecalage} comonad $\mathsf{D}_m$ on
the arrow category $\cat{Cat}^\atwo$ whose action on objects is given
by
\begin{equation}\label{eq:1}
  \E \xrightarrow{P} \C \qquad \mapsto \qquad \textstyle\sum_{Y \in 
    \E} \E / Y \xrightarrow{\Sigma_{Y \in \C} P/Y} \sum_{Y \in \E}
  \C / PY\rlap{ .}
\end{equation}

To relate this to operadic categories, we consider those objects of
$\cat{Cat}^\atwo$ which are obtained to within isomorphism as the
canonical projection $P_\C \colon \E_\C \rightarrow \C$ from the
\emph{category of elements} of a functor
$\abs{\thg} \colon \C \rightarrow \S$. These objects span a full
subcategory of $\cat{Cat}^\atwo$ which is equivalent to the lax slice
category $\cat{Cat} /\!/ \cat{S}$; they are moreover closed under the
action of $\mathsf{D}_m$, which thus restricts back to a comonad on
$\cat{Cat}/\!/\S$. Now by following the same trajectory as the unary
case, starting from this comonad on $\cat{Cat}/\!/\S$, we already come
very close to characterising operadic categories.

The first point to make is that an object
$(\C, \abs{\thg} \colon \C \rightarrow \S) \in \cat{Cat}/\!/\S$ is a
$\mathsf{D}_m$-coalgebra just when $\C$ has chosen terminal objects in
each connected component, and $\abs{\thg}$ sends each of these to
$1 \in \S$. Since these are among the requirements for an operadic
category, every operadic category gives rise to a
$\mathsf{D}_m$-coalgebra.

Like before, we can ask what it means to equip such a 
$\mathsf{D}_m$-coalgebra with algebra structure for the induced monad
$\mathsf{\tilde D}_m$ on $(\cat{Cat}/\!/\S)^{\mathsf{D}_m}$. The
action of this monad on $(\C, \abs{\thg})$ is given by the category
$\sum_{X \in \C, i \in \abs X} \C / X$ of~\eqref{eq:9}, endowed with a
suitable functor to $\S$; therefore, the basic datum of
$\mathsf{\tilde D}_m$-algebra structure is a functor of the same form
as~\eqref{eq:9}. This seems promising, but what  we find is:

\begin{Thm*}
  The category of algebras for the modified d\'ecalage monad
  $\mathsf{\tilde D}_m$ on $(\cat{Cat}/\!/\S)^{\mathsf{D}_m}$ is
  isomorphic to the category of lax-operadic categories.
\end{Thm*}

Here, a \emph{lax-operadic category} is a new notion, which
generalises that of operadic category by replacing the assertion of
equalities $\abs{f}^{-1}(i) = \abs{f^{-1}(i)}$ on cardinalities of
abstract fibres with a collection of coherent functions
$\abs{f}^{-1}(i) \rightarrow \abs{f^{-1}(i)}$. Since it is not yet
clear that this extra generality has any practical merit, our final
objective is to find a version of the above result which removes the
qualifier ``lax''.

The source of the laxity is easy to pinpoint. A
$\mathsf{\tilde D}_m$-algebra structure is given by a map
$D_m(\C, \abs{\thg}) \rightarrow (\C, \abs{\thg})$ in
$\cat{Cat} /\!/ \S$, whose data involves not only a
functor~\eqref{eq:9}, but also a natural transformation relating the
functors to $\S$. The components of this natural transformation are
the comparison functions
$\abs{f}^{-1}(i) \rightarrow \abs{f^{-1}(i)}$, so that the genuine
operadic categories correspond to those
$\mathsf{\tilde D}_m$-algebras whose structure map is given by a
\emph{strictly} commuting triangle over $\S$.

However, we cannot simply restrict the modified d\'ecalage comonad
$\mathsf{D}_m$ from the lax slice category $\cat{Cat} /\!/ \S$ back to
the strict slice category $\cat{Cat} / \S$, and then proceed as
before. The problem is that $\mathsf{D}_m$ does not restrict, since
the counit maps $\varepsilon_\C \colon D_m(\C) \rightarrow \C$ in
$\cat{Cat} /\!/ \S$ involve triangles which are genuinely
lax-commutative. On the other hand, it turns out that
we \emph{can}
restrict the lifted monad $\mathsf{\tilde D}_m$ on
$(\cat{Cat} /\!/ \S)^{\mathsf{D}_m}$ back to the subcategory
$(\cat{Cat}^\mathsf{D}) / \S$ on the strictly commuting triangles.
Having done so, our final result quickly follows:

\begin{Thm*}
  The category of algebras for the modified d\'ecalage monad
  $\mathsf{\tilde D}_m$ on $\cat{Cat}^\mathsf{D}/\S$ is isomorphic
  to the category of operadic categories.
\end{Thm*}

The rest of this article will fill in the details of the above sketch.
The plan is quite simple. In Section~\ref{sec:operadic-categories}, we
recall Batanin and Markl's definition of operadic
category~\cite{Batanin2015Operadic}; then in
Section~\ref{sec:decalage} we recall the d\'ecalage construction and
establish the first of the two links with the notion of operadic
category. In Section~\ref{sec:line-oper-categ}, we prove our first
main theorem, characterising unary operadic categories in terms of
d\'ecalage. Section~\ref{sec:modified-decalage} is devoted to
describing the modified d\'ecalage construction required to capture
general operadic categories. Finally, in
Sections~\ref{sec:char-lax-oper} and~\ref{sec:char-oper-categ}, we
prove our second and third theorems, giving the characterisations of
lax-operadic categories and, finally, of operadic categories
themselves.

\section{Operadic categories}
\label{sec:operadic-categories}

We begin with some necessary preliminaries. We say that a category
$\C$ is \emph{endowed with local terminal objects} if each connected
component of $\C$ is provided with a chosen terminal object; we write
$uX$ for the chosen terminal in the connected component of $X \in \C$
and $\tau_X \colon X \rightarrow uX$ for the unique map.

We write $\S$ for the category whose objects are the sets
$\underline n = \{1, \dots, n\}$ for $n \in \mathbb{N}$ and whose maps
are arbitrary functions. Note that $\S$ has a \emph{unique} terminal
object $\underline 1$ which we use to endow $\S$ with local terminal
objects; we may also sometimes write the unique element of
$\underline 1$ as $\ast$ rather than $1$.

Given
$\varphi \colon \underline m \rightarrow \underline n$ in $\S$ and
$i \in \underline n$, there is a unique monotone injection
\begin{equation}\label{eq:10}
  \varepsilon_{\varphi, i} \colon \varphi^{-1}(i) \rightarrow
  \underline m
\end{equation}
in $\S$ whose image is $\{\,j \in \underline m : \varphi(j) = i\,\}$;
we call the object $\varphi^{-1}(i)$ the \emph{fibre of $\varphi$ at
  $i$}. Given also
$\psi \colon \underline \ell \rightarrow \underline m$ in $\S$, we
write $\psi^\varphi_{i}$ for the unique map of $\S$ rendering
\begin{equation}\label{eq:11}
  \cd{
    {(\varphi\psi)^{-1}(i)} \ar[r]^-{\psi^\varphi_{i}}
    \ar[d]_{\varepsilon_{\varphi \psi, i}} &
    {\varphi^{-1}(i)} \ar[d]^{\varepsilon_{\varphi, i}}  \\
    {\underline \ell} \ar[r]^-{\psi}
     &
    {\underline m} 
  }
\end{equation}
commutative, and call it the \emph{fibre map of $\psi$ with respect to
  $\varphi$ at $i$}.

The Batanin--Markl notion of operadic category which we now reproduce
can be seen as specifying a category with \emph{formal} notions of
fibre and fibre map. The fibres of a map need not be
subobjects of the domain as in the case of $\mathcal{S}$, but the
axioms will ensure that they retain many important properties of
fibres in $\mathcal{S}$.

\begin{Defn}
  \label{def:operadic-category}\cite{Batanin2015Operadic}
  An \emph{operadic category} is given by the following data:
  \begin{enumerate}[label=(D\tstyle{\arabic*})]
  \item \label{data:operadic-Q1} A category $\C$ endowed with local
    terminal objects;
    \item \label{data:operadic-Q2} A \emph{cardinality functor}
      $\abs{\thg} \colon \C \to \S$;
    \item \label{data:operadic-Q3} For each object $X\in \C$ and each
      $i \in \abs{X}$ a \emph{fibre functor}
      $$\fibre_{X,i} \colon \C/X \to \C$$
      whose action on objects and morphisms we denote as follows:
      \begin{align*}
        \cd[@C1.3em]{
          Y \ar[rr]^-{f} && X
        } \qquad &\mapsto \qquad f^{-1}(i)\\
        \cd[@C1em@R-0.7em]{{Z} \ar[rr]^-{g} \ar[dr]_-{fg} & &
          {Y} \ar[dl]^-{f} \\ &
          {X}} \qquad &\mapsto \qquad
        g^f_{i} \colon (fg)^{-1}(i) \to f^{-1}(i)\rlap{ ,}
      \end{align*}
      referring to the object $f^{-1}(i)$ as the \emph{fibre
      of $f$ at $i$}, and the morphism
    $g^f_{i} \colon (fg)^{-1}(i) \to f^{-1}(i)$ as the \emph{fibre
      map of $g$ with respect to $f$ at~$i$};
  \end{enumerate}
  all subject to the following axioms, where
  in~\ref{axQ:BM-fibres-of-local-fibres}, we write $\varepsilon j$
  for the image of $j \in {\abs f}^{-1}(i)$ under the map
  $\varepsilon_{\abs f, i} \colon {\abs f}^{-1}(i) \rightarrow \abs Y$
  of~\eqref{eq:10}:

  \begin{enumerate}[label=(A\tstyle{\arabic*})]
  \item \label{axQ:BM-abs(lt)} If $X$ is a local terminal then
    $\abs{X}=\underline 1$;
  \item \label{axQ:BM-fibres-of-identities} For all $X \in \C$ and
    $i \in \abs X$, the object $(1_X)^{-1}(i)$ is chosen terminal;
  \item \label{axQ:BM-67} For all $f \in \C / X$ and $i \in \abs X$,
    one has $\abs{\smash{f^{-1}(i)}} = {\abs f}^{-1}(i)$, while for
    all $g \colon fg \rightarrow f$ in $\C / X$ and $i \in \abs X$,
    one has $\abs{\smash{g^f_{i}}} = \smash{\abs{g}^{\abs f}_{i}}$;
  \item \label{axQ:BM-fibres-of-tau-maps} For $X \in \C$, one has
    $\tau_X^{-1}(\ast) = X$, and for $f \colon Y \to X$, one has
    $f^{\tau_X}_\ast = f$;
  \item \label{axQ:BM-fibres-of-local-fibres} For
    $g \colon fg \rightarrow f$ in $\C/X$, $i \in \abs X$ and
    $j \in \abs{f}^{-1}(i)$, one has that
    $(g^f_i)^{-1}(j) = g^{-1}(\varepsilon j)$, and given also
    $h \colon fgh \rightarrow fg$ in $\C / X$, one has
    $(h^{fg}_i)^{g^f_i}_{j} = h^g_{\varepsilon j}$.
  \end{enumerate}
  
  A functor $F \colon \C \rightarrow \someothercat$ between operadic
  categories is called an \emph{operadic functor} if it strictly
  preserves local terminal objects, strictly commutes with the
  cardinality functors to $\S$, and preserves fibres and fibre maps in
  the sense that
  \begin{equation*}
    F(f^{-1}(i)) = (Ff)^{-1}(i) \qquad \text{and} \qquad F(g^f_{i}) =
    (Fg)^{Ff}_{i}
  \end{equation*}
  for all $g \colon fg \rightarrow f$ in $\C/X$ and $i \in \abs X$. We
  write $\cat{Op\C at}$ for the category of operadic categories and
  operadic functors.
\end{Defn}

The preceding definitions are exactly those
of~\cite{Batanin2015Operadic} with only some minor notational changes
for clarity. The most substantial of these is that we make explicit
the use of the monotone injections~\eqref{eq:10} in the
axiom~\ref{axQ:BM-fibres-of-local-fibres}, whereas
in~\cite{Batanin2015Operadic} this is left implicit. In light of this,
let us spend a moment doing the necessary type-checking to see that
this axiom makes sense.

Intuitively, the first clause of \ref{axQ:BM-fibres-of-local-fibres}
identifies the fibres of the fibre maps of a map, with the fibres of
that map. Therein we have $g^f_i \colon (fg)^{-1}(i) \to f^{-1}(i)$,
and $j \in {\abs f}^{-1}(i) = \abs{\smash{f^{-1}(i)}}$, so that one
can consider the object $(g^f_i)^{-1}(j)$. On the other hand,
we have $\varepsilon j \in \abs Y$ and $g \colon Z \rightarrow Y$ and
so can equally consider the object $g^{-1}(\varepsilon j)$; now the
first part of \ref{axQ:BM-fibres-of-local-fibres} states that these
two are equal.

As for the second part of \ref{axQ:BM-fibres-of-local-fibres},
this says that the fibre maps of the fibre maps of a map, are
themselves fibre maps of that map. In this case, functoriality of the
fibre functor $\varphi_{X,i} \colon \C / X \rightarrow \C$ implies
that we have an equality
\begin{equation*}
  (fgh)^{-1}(i) \xrightarrow{(gh)^f_i} f^{-1}(i) \quad = \quad (fgh)^{-1}(i) \xrightarrow{h^{fg}_i} (fg)^{-1}(i)
  \xrightarrow{g^f_i} f^{-1}(i)
\end{equation*}
and again we have $j \in \abs{f}^{-1}(i) = \abs{f^{-1}(i)}$.
It follows that the fibre map of $h_{i}^{fg}$ with respect to $g_i^f$
at $j$ is given as to the left in:
\begin{equation*}
  (h^{fg}_i)^{g^f_i}_{j} \colon ((gh)^f_i)^{-1}(j)
  \rightarrow (g^f_i)^{-1}(j) \qquad \quad h^g_{\varepsilon j} \colon
  (gh)^{-1}(\varepsilon j) \rightarrow g^{-1}(\varepsilon j)\rlap{ .}
\end{equation*}
On the other hand, one could just consider the fibre map of $h$ with
respect to $g$ at $\varepsilon j \in \abs Y$, as to the right. The
first part of~\ref{axQ:BM-fibres-of-local-fibres} assures us that the
domains and codomains of these maps coincide, and now the second part
asserts that the maps themselves are equal.

\begin{Ex}
  \label{ex:1}
  The most basic example of an operadic category is $\S$ itself. The
  choice of local terminals is the unique one, the cardinality functor
  is the identity, and the action of the fibre functors is defined as
  in~\eqref{eq:10} and~\eqref{eq:11}.
\end{Ex}

Many more examples of operadic categories are discussed
in~\cite{Batanin2015Operadic}; we give here two new examples inspired
by probability theory.

\begin{Ex}
  \label{ex:3}
  Let $\C$ be the category of finite sub-probability spaces. Its
  objects are lists $r = (r_1, \dots, r_n)$ where each $r_i \in [0,1]$
  and $\Sigma_i r_i \leqslant 1$; its maps
  $\varphi \colon (s_1, \dots, s_m) \rightarrow (r_1, \dots, r_n)$ are
  maps $\varphi \colon \underline m \rightarrow \underline n$ of $\S$
  such that $r_i = \Sigma_{j \in \varphi^{-1}(i)} s_i$. There is an
  obvious cardinality functor $\abs{\thg} \colon \C \rightarrow \S$,
  and a \emph{unique} choice of local terminals: indeed, $\C$ is a
  coproduct of categories $\C = \Sigma_{r \in [0,1]} \C_r$ where
  $\C_r$ has the unique terminal object $(r)$. For the abstract
  fibres, given $\varphi \colon s \rightarrow r$ in $\C$ and
  $i \in \abs{r}$, we define $\varphi^{-1}(i)$ to be
  $(s_{\varepsilon 1}, \dots, s_{\varepsilon k})$ where
  $k = \abs{\varphi}^{-1}(i)$ and
  $\varepsilon = \varepsilon_{\abs{\varphi},i} \colon
  \abs{\varphi}^{-1}(i) \rightarrow \abs{s}$ is as in~\eqref{eq:10};
  finally, fibre maps in $\C$ are as in~$\S$.
\end{Ex}

\begin{Ex}
  \label{ex:5}
  Let $\C_1$ be the category of finite probability spaces, i.e., the
  connected component of $(1)$ in the category $\C$ of the previous
  example. This subcategory of $\C$ is not a sub-operadic category;
  however, it bears a different operadic category structure which
  describes \emph{disintegration} of finite probability measures.

  We begin with the cardinality functor
  $\abs{\thg}_1 \colon \C_1 \rightarrow \S$. For any
  $r = (r_1, \dots, r_n) \in \C_1$, we let $(r_{p_1}, \dots, r_{p_k})$
  be the sublist of $r$ obtained by deleting all zeroes, and now take
  $\abs{r}_1 = \underline k$. Given a map
  $\varphi \colon s \rightarrow r$ in $\C_1$, where $s$ has sublist
  $(s_{q_1}, \dots, s_{q_\ell})$ of non-zero entries, we determine
  $\abs{\varphi}_1 \colon \abs{s}_1 \rightarrow \abs{r}_1$ by
  requiring that $\abs{\varphi}(q_j) = p_{\abs{\varphi}_1(j)}$; i.e.,
  $\abs{\varphi}_1$ is the restriction of $\abs{\varphi}$ to the
  indices of non-zero entries.

  To define the $\C_1$-fibres, we employ the \emph{normalisation} of a
  non-zero sub-probability space
  $r = (r_1, \dots, r_n) \in \C \setminus \C_0$; this is the
  probability space $\overline r \in \C_1$ with
  $\overline r = (r_1 / \Sigma_i r_i, \dots, r_n / \Sigma_i r_i)$. Now
  given $\varphi \colon s \rightarrow r$ in $\C_1$ and
  $i \in \abs{r}_1$, we define the $\C_1$-fibre $\varphi^{-1}(i)$ to
  be the normalisation of the $i$th non-zero $\C$-fibre
  $\varphi^{-1}(p_i)$. Note that we \emph{cannot} normalise a
  $\C$-fibre $\varphi^{-1}(j)$ for which $r_j = 0$; this is why we had
  to remove such $j$ in defining the cardinality functor
  $\abs{\thg}_1$. Finally, given also $\psi \colon t \rightarrow s$ in
  $\C_1$, we define the $\C_1$-fibre map $\psi^{\varphi}_i$ to have
  underlying $\S$-map $\psi^{\varphi}_{p_i}$.
\end{Ex}

\section{D\'ecalage}
\label{sec:decalage}

In this section, we recall the d\'ecalage comonad on $\cat{Cat}$,
characterise its category of coalgebras, and explain how this links up
with the notion of operadic category. Throughout, ``d\'ecalage'' 
will always mean \emph{upper} d\'ecalage.

The d\'ecalage comonad on $\cat{Cat}$ can be obtained as a restriction
of Illusie's d\'ecalage comonad~\cite{Illusie1972Complexe} on
$[\Delta^\mathrm{op}, \cat{Set}]$, the category of simplicial sets.
This is, in turn, obtained from the monad
$\mathsf{T} = (T, \eta, \mu)$ on the category $\Delta$ of non-empty
finite ordinals and monotone maps given by freely adjoining a top
element. In terms of the usual presentation of $\Delta$ in terms of
``coface'' maps $\delta_i$ and ``codegeneracy'' maps $\sigma_j$, this
monad is given by the data:
\begin{equation*}
  {T[n] = [n+1]}\text{,} \qquad
  {T\delta_i = \delta_i}\text{,} \qquad
  {T\sigma_j = \sigma_j}\text{,} \qquad
  {\eta_{[n]} = \delta_{n+1}}\ \ \text{ and }\ \ 
  {\mu_{[n]} = \sigma_{n+1}}\rlap{ .}
\end{equation*}
It follows that $\mathsf{T}^\mathrm{op}$ is a \emph{co}monad on
$\Delta^\mathrm{op}$, so that precomposition with
$\mathsf{T}^\mathrm{op}$ is a comonad on
$[\Delta^\mathrm{op}, \cat{Set}]$; this is the d\'ecalage comonad.

The classical \emph{nerve} functor
$\mathrm{N} \colon \cat{Cat} \rightarrow [\Delta^\mathrm{op},
\cat{Set}]$ exhibits the category of small categories as equivalent to
a full subcategory of simplicial sets. The simplicial sets in this
full subcategory happen to be closed under the action of the
d\'ecalage comonad, which thereby restricts to a comonad $\mathsf{D}$
on $\cat{Cat}$. The underlying endofunctor $D$ of this comonad sends a
category $\C$ to the coproduct of its slices:
\begin{equation}\label{eq:DecCoprodSlices}
  D(\C) = \textstyle \sum_{X \in \C} \C / X\rlap{ ;}
\end{equation}
the counit $\varepsilon_\C \colon D(\C) \rightarrow \C$ is the
\emph{copairing} of the domain projections $\C / X \rightarrow \C$
from the slices (i.e., the map induced from the family of domain
projections by the universal property of coproduct); while the
comultiplication $\delta_\C \colon D(\C) \rightarrow DD(\C)$, which is
a functor
\begin{equation*}
  \delta_\C \colon \textstyle \sum_{X \in \C} \C / X \rightarrow
  \sum_{f \in D(\C)} D(\C) / f\rlap{ ,}
\end{equation*}
sends the $X$-summand to the $1_X$-summand via the isomorphism
$\C / X \rightarrow D(\C) / 1_X$.

We now characterise the category of coalgebras for the d\'ecalage
comonad as the category $\cat{Cat}_{\ell t}$  whose objects
are small categories endowed with local terminal objects, and whose
morphisms are functors which preserve chosen local terminals.

\begin{Prop}\label{prop:D-CoAlg}
  The category $\cat{Cat}^\mathsf{D}$ of $\mathsf{D}$-coalgebras is
  isomorphic to $\cat{Cat}_{\ell t}$ over $\cat{Cat}$. Under this
  isomorphism, the $\mathsf{D}$-coalgebra structure on
  $\C \in \cat{Cat}_{\ell t}$ is given by the functor
  $\tau \colon \C\to D(\C)$ which takes $X \in \C$ to
  $\tau_X \colon X \to uX \in D(\C)$.
\end{Prop}

\begin{proof}
  It suffices to show that the forgetful functor
  $U \colon \cat{Cat}_{\ell t} \rightarrow \cat{Cat}$ is strictly
  comonadic, and that the induced comonad is isomorphic to
  $\mathsf{D}$. Towards the first of these, it is clear that $U$
  strictly creates limits and is faithful, and so by the Beck theorem
  will be strictly comonadic so long as it has a right adjoint.

  We can endow the category $D(\C)$ with the chosen terminal object
  $1_X$ in each connected component $\C / X$, so making it into an
  object of $\cat{Cat}_{\ell t}$; we claim this gives the value at
  $\C$ of the desired right adjoint. Thus, for any
  $\B \in \cat{Cat}_{\ell t}$ and functor
  $F \colon \B \rightarrow \C$, we must exhibit a unique factorisation
  \begin{equation}\label{eq:3}
    F = \B \xrightarrow{G} D(\C) \xrightarrow{\varepsilon_\C} \C
  \end{equation}
  where $G$ strictly preserves chosen local terminals. Such a $G$ must
  send each object $X \in \C$ to an object of $D(\C)$ with domain
  projection $FX$. In particular, each chosen terminal $uX$ of $\B$
  must be sent to a chosen terminal of $D(\C)$ with domain $FuX$, and
  so we must have $G(uX) = 1_{FuX}$. Furthermore, such a $G$, if it
  exists, must send each map $f \colon Y \rightarrow X$ of $\B$ to a
  map in $D(\C)$ as to the left in:
  \begin{equation*}
    \cd[@!C@C-1em]{
      {FY} \ar[rr]^-{Ff} \ar[dr]_-{GY\,} & &
      {FX} \ar[dl]^-{\, GX} &&
      {FY} \ar[rr]^-{F\tau_X} \ar[dr]_-{GX} & &
      {FuX}\rlap{ .} \ar[dl]^-{\, 1_{FuX}} \\ &
      {\bullet} & & & & 
      {FuX}
    }
  \end{equation*}
  In particular, taking $f = \tau_X$ yields the commuting triangle to
  the right, so that on objects we must have
  $GX = (F\tau_X \colon FX \rightarrow FuX)$. So $G$ is unique if it
  exists; but it easy to see that defining $G$ in this way does indeed
  yield a map $G \colon \B \rightarrow D(\C)$ in $\cat{Cat}_{\ell t}$
  preserving chosen terminals and factorising~\eqref{eq:3} as
  required.

  So $U \colon \cat{Cat}_{\ell t} \rightarrow \cat{Cat}$ has a right
  adjoint $R$, and by strict comonadicity, $\cat{Cat}_{\ell t}$ is
  isomorphic to the category of $UR$-coalgebras. By construction, the
  underlying functor and counit of $UR$ are \emph{equal} to $D$ and
  $\varepsilon$, while the comultiplication at $\C$ is the unique
  factorisation~\eqref{eq:3} of
  $F = 1_{D(\C)} \colon D(\C) \rightarrow D(\C)$ through a map in
  $\cat{Cat}_{\ell t}$. As $\delta_\C \colon D(\C) \rightarrow DD(\C)$
  is easily seen to be such a factorisation, we conclude that
  $\mathsf{D} = UR$ and so
  $\cat{Cat}_{\ell t} \cong \cat{Cat}^\mathsf{D}$ as required.
\end{proof}

To motivate the developments which will follow, we now establish a
first link between operadic categories and d\'ecalage, by showing how
the data and axioms for an operadic category can be partially
re-expressed in terms of structure in
$\cat{Cat}_{\ell t} \cong \cat{Cat}^\mathsf{D}$. Of
course,~\ref{data:operadic-Q1} asserts that $\C$ is an object in
$\cat{Cat}_{\ell t}$, whereupon axiom~\ref{axQ:BM-abs(lt)} asserts
that the cardinality functor $\abs{\thg} \colon \C \rightarrow \S$ is
a map therein. Similarly, axiom~\ref{axQ:BM-fibres-of-identities}
states that each functor $\varphi_{X,i} \colon \C / X \rightarrow \C$
is a map of $\cat{Cat}_{\ell t}$, where we take the chosen (local)
terminal object in $\C / X$ to be the identity $1_X$.

To express~\ref{axQ:BM-67}, we define for each $X \in \C$ and
$i \in \abs X$ a cardinality functor
$\abs{\thg}_{X,i} \colon \C / X \rightarrow \S$ as the composite of
$\abs{\thg}/X \colon \C / X \rightarrow \S / \abs{X}$ with the fibre
functor $\varphi_{\abs X, i} \colon \S / \abs X \rightarrow \S$ of the
operadic category $\S$; thus, on objects,
$\abs{f}_{X,i} = \smash{\abs{f}^{-1}(i)}$.
Now~\ref{axQ:BM-67} asserts that the following diagram commutes for
all $X \in \C$, $i \in \abs X$:
\begin{equation}\label{axiomstriangle}
  \cd[@!C@C-2em@R-0.8em]{
    \C/X \ar[rr]^{\fibre_{X,i}} \ar[rd]_{\abs{\thg}_{X,i}} &&
    \C \rlap{ .}\ar[ld]^{\abs{\thg}} \\
    & \mathcal{S} &
  }
\end{equation}

We may express all of the above more compactly as follows. For any
object $\abs{\thg}_\C \colon \C \rightarrow \S$ of
$\cat{Cat}_{\ell t} / \S$, we write $D_m(\C)$ for the category
$\Sigma_{X \in \C, i \in \abs X} \C / X$, seen as an object of
$\cat{Cat}_{\ell t}$ by choosing each identity map as a local
terminal, and write
$\smash{\abs{\thg}_{D_m(\C)}} \colon \smash{D_m(\C)} \rightarrow \S$
for the copairing of the maps
$\abs{\thg}_{X,i} \colon \C / X \rightarrow \S$. Now to give the
data~\ref{data:operadic-Q1}--\ref{data:operadic-Q3} and axioms
\ref{axQ:BM-abs(lt)}--\ref{axQ:BM-67} for an operadic category is to
give an object $(\C, \abs{\thg}_\C)$ of $\cat{Cat}_{\ell t} / \S$ and
a map
$\varphi \colon (D_m(\C), \abs{\thg}_{D_m(\C)}) \rightarrow (\C,
\abs{\thg}_{\C})$.

It remains to account for axioms~\ref{axQ:BM-fibres-of-tau-maps}
and~\ref{axQ:BM-fibres-of-local-fibres}. In fact, it turns out that
the assignation
$\smash{(\C, \abs{\thg}_\C) \mapsto (D_m(\C), \abs{\thg}_{D_m(\C)})}$
is the action on objects of a monad $\mathsf{\tilde D}_m$ on the
category $\cat{Cat}_{\ell t} / \S$, and that the remaining axioms are
just those needed for
$\varphi \colon (D_m(\C), \abs{\thg}_{D_m(\C)}) \rightarrow (\C,
\abs{\thg}_{\C})$ to endow $(\C, \abs{\thg})$ with
$\mathsf{\tilde D}_m$-algebra structure. While we could verify this
straight away in a hands-on fashion, we prefer to give an argument
which justifies the constructions in terms of a deeper link to the
d\'ecalage construction. In the end,
the claimed monad structure on $\mathsf{\tilde D}_m$ will be exhibited
in Definition~\ref{def:1} below, and the characterisation of its
algebras as operadic categories given in Theorem~\ref{thm:4}.


\section{Characterising unary operadic categories}
\label{sec:line-oper-categ}

The characterisation of general operadic categories in terms of
d\'ecalage will require a modification of the d\'ecalage construction,
to be introduced in Section~\ref{sec:modified-decalage} below. As a
warm-up for this, we consider the case of \emph{unary} operadic
categories, for which the usual d\'ecalage will suffice.

\begin{Defn}
  \label{def:2}
  An operadic category is \emph{unary} if $\abs X = \underline 1$ for
  all $X \in \C$. We write $\cat{Op\C at}_1$ for the category of unary
  operadic categories and operadic functors.
\end{Defn}

\begin{Ex}
  \label{ex:2}
  For any category $\C$, the category $D(\C) = \sum_{X\in \C}\C/X$ is a unary
  operadic category. The chosen local terminals are the identity maps,
  and the unique fibre of a map $g \colon fg \rightarrow f$ is the
  object $g$. Given another map $h \colon fgh \rightarrow fg$, the
  fibre map of $h$ with respect to $g$ at $\ast$ is taken to be
  $h \colon gh \rightarrow g$. 
\end{Ex}

\begin{Ex}
  \label{ex:4}
  If $\C$ is a pointed category with a chosen zero object and chosen
  kernels, we can attempt to impose a unary operadic structure as
  follows: the chosen (local) terminal is the zero object; the unique
  fibre of a map $f \colon Y \rightarrow X$ is its kernel; and the
  fibre map of $g \colon Z \rightarrow Y$ with respect to $f$ is the
  restriction $\res g {\ker{fg}} \colon \ker{fg} \rightarrow \ker{f}$.
  However, whether these data satify the required axioms is sensitive
  to the choice of kernels. For instance, if
  $g \colon Z \rightarrow Y$ and $f \colon Y \rightarrow X$, then the
  chosen kernel of $g$, though always \emph{isomorphic} to the chosen
  kernel of $\res g{\ker{fg}} \colon \ker{fg} \rightarrow \ker{f}$,
  need not be \emph{equal} to it as required by
  axiom~\ref{axQ:BM-fibres-of-local-fibres}.

  Often, there \emph{is} an appropriate choice of kernels; for example
  if $\C$ is $\cat{Set}_\ast$ or $\cat{Ab}$ or $k\text-\cat{Vect}$ or
  $\cat{Ch}(R\text-\cat{Mod})$, then we can take the kernel of any
  identity map to be the chosen zero object, and the kernel of any
  other map to be given by the usual subset formula; this yields the
  necessary axioms for a unary operadic category.

  Yet even for a $\C$ where we cannot choose kernels appropriately, we
  can always consider the \emph{equivalent} category
  $\cat{Pt}(\C^\mathrm{op}, \cat{Set}_\ast)_{\mathrm{rep}}$ of
  representable zero-preserving functors to $\cat{Set}_\ast$, and
  endow this with unary operadic structure given pointwise as in
  $\cat{Set}_\ast$. Note that this structure need \emph{not} transport
  back to an operadic structure on $\C$, since the notion of operadic
  category is not invariant under equivalence (in the terminology
  of~\cite{Blackwell1989Two-dimensional} it is not \emph{flexible}).
\end{Ex}

In the unary case, we can effectively ignore the cardinality functor
down to $\S$; so on repeating the analysis at the end of the preceding
section, we find that the data and first three axioms for a unary
operadic category $\C$ are encoded precisely by a map
$D(\C) \rightarrow \C$ in $\cat{Cat}_{\ell t}$. To complete this
analysis, we will show that the assignation $\C \mapsto D(\C)$
underlies a monad on $\cat{Cat}_{\ell t}$ whose category of algebras
is isomorphic to $\cat{Op\C at}_1$. The monad structure arises as follows.


\begin{Defn}
  \label{def:4}
  The \emph{d\'ecalage monad}
  $\mathsf{\tilde D} = (\tilde D, \eta, \mu)$ on $\cat{Cat}_{\ell t}
  \cong \cat{Cat}^\mathsf{D}$
  is the monad induced by the forgetful--cofree adjunction
  $\cat{Cat}^\mathsf{D} \leftrightarrows \cat{Cat}$.
\end{Defn}

Since the proof of Proposition~\ref{prop:D-CoAlg} furnishes us with an
explicit description of the forgetful--cofree adjunction
$\cat{Cat}^\mathsf{D} \leftrightarrows \cat{Cat}$, we can read off
from it the following description of the d\'ecalage monad:
\begin{enumerate}[(i),itemsep=0.25\baselineskip,widest=iii]
\item The underlying functor $\tilde D$ on objects sends $\C$ to
  $\sum_{X \in \C} \C / X$ endowed with the local terminal objects
  $1_X \in \C / X$; while on morphisms, it sends
  $F \colon \C \rightarrow \someothercat$ to the functor which maps
  the $X$-summand of $\sum_{X \in \C} \C / X$ to the $FX$-summand of
  $\sum_{Y \in \someothercat} \someothercat / Y$ via
  $F/X \colon \C / X \rightarrow \someothercat / FX$;
\item The unit map $\eta_\C \colon \C \rightarrow \tilde D(\C)$ is
  defined on objects by $\eta_\C(X) = \tau_X \colon X \rightarrow uX$
  and on morphisms by
  $\eta_\C(f \colon Y \rightarrow X) = f \colon \tau_Y \rightarrow
  \tau_X$;
\item The multiplication map
  $\mu_\C \colon \tilde D \tilde D (\C) \rightarrow \tilde D(\C)$,
  which is given by a functor
  $ \textstyle\sum_{f \in D(\C)} D(\C) / f \rightarrow \sum_{X \in \C}
  \C / X $, sends the summand indexed by $f \colon Y \rightarrow X$ to
  the summand indexed by $Y$ via the isomorphism
  $D(\C)/f \rightarrow \C / Y$.
\end{enumerate}

Using this description, we can now prove our first main theorem.

\begin{Thm}
  \label{thm:1}
  The category of algebras for the d\'ecalage monad
  $\mathsf{\tilde{D}}$ on
  $\cat{Cat}_{\ell t} \cong \cat{Cat}^\mathsf{D}$ is isomorphic to the
  category $\cat{Op \Cat}_1$ of unary operadic categories.
\end{Thm}

\begin{proof}
  We have already argued that the data and first three axioms for
  a unary operadic category $\C$ are encapsulated by giving
  the object $\C \in \cat{Cat}_{\ell t}$ together with the map
  $\varphi \colon \tilde D(\C) \rightarrow \C$ in $\cat{Cat}_{\ell t}$
  obtained as the copairing of the fibre functors
  $\varphi_{X,\ast} \colon \C / X \rightarrow \C$.
  Given this, we can read off from Definition~\ref{def:4}
  that~\ref{axQ:BM-fibres-of-tau-maps} asserts precisely the unit
  axiom $\varphi \circ \eta_\C = 1_\C$, and
  that~\ref{axQ:BM-fibres-of-local-fibres} asserts the multiplication
  axiom
  $\varphi \circ \mu_\C = \varphi \circ \tilde D(\varphi) \colon
  \tilde D \tilde D(\C) \rightarrow \C$. So
  $\mathsf{\tilde D}$-algebras in $\cat{Cat}_{\ell t}$ are in
  bijection with unary operadic categories; the corresponding
  bijection on maps is direct.
\end{proof}

Using this result, we may obtain a further description of unary
operadic categories which, though not necessary for the subsequent
results of this paper, is nonetheless enlightening. We observed above
that the d\'ecalage comonad on $\cat{Cat}$ is the restriction along
the full inclusion
$\mathrm{N} \colon \cat{Cat} \rightarrow [\Delta^\mathrm{op},
\cat{Set}]$ of the d\'ecalage comonad on simplicial sets. It follows
that we have a full inclusion
\begin{equation}\label{eq:23}
  \cat{Op\C at} \xrightarrow{\cong}
  (\cat{Cat}^\mathsf{D})^{\mathsf{\tilde D}}
  \xrightarrow{(\mathrm{N}^\mathsf{D})^\mathsf{\tilde D}}
  ([\Delta^\mathrm{op}, \cat{Set}]^\mathsf{D})^{\mathsf{\tilde D}}
\end{equation}
(where we re-use the notation $\mathsf{D}$ and
$\mathsf{\tilde D}$ for the d\'ecalage comonad on
$[\Delta^\mathrm{op}, \cat{Set}]$ and the induced monad on
$[\Delta^\mathrm{op}, \cat{Set}]^\mathsf{D}$) whose essential image
comprises just those $\mathsf{\tilde D}$-algebras in
$[\Delta^\mathrm{op}, \cat{Set}]^\mathsf{D}$ whose underlying
simplicial set satisfies the Segal condition. On the other hand, we
have a straightforward characterisation of the category
$([\Delta^\mathrm{op}, \cat{Set}]^\mathsf{D})^{\mathsf{\tilde D}}$:

\begin{Lemma}
  \label{lem:3}
  The comparison functor
  \begin{equation*}
    [\Delta^\mathrm{op}, \cat{Set}] \rightarrow ([\Delta^\mathrm{op},
    \cat{Set}]^{\mathsf{D}})^{\mathsf{\tilde D}}
  \end{equation*}
  sending a simplicial set $X$ to $D(X)$ with its canonical
  $\mathsf{\tilde D}$-algebra structure, is an equivalence of categories.
\end{Lemma}
\begin{proof}
  The functor part of the comonad $\mathsf{D}$ on
  $[\Delta^\mathrm{op}, \cat{Set}]$ is given by precomposition with
  $T^\mathrm{op} \colon \Delta^\mathrm{op} \rightarrow
  \Delta^\mathrm{op}$, and so is cocontinuous. Thus, for the
  forgetful--cofree adjunction
  \begin{equation*}
    \cd{
      [\Delta^\mathrm{op}, \cat{Set}]^\mathsf{D} \ar@<-4.5pt>[r]_-{U^\mathsf{D}} \ar@{<-}@<4.5pt>[r]^-{G^\mathsf{D}} \ar@{}[r]|-{\top} &
      {[\Delta^\mathrm{op}, \cat{Set}]} 
    }
  \end{equation*}
  the functor $G^\mathsf{D}$ is again cocontinuous. Moreover,
  $U^\mathsf{D}$ is conservative, and it is easy to see that
  $U^\mathsf{D}G^\mathsf{D} = D$ is conservative---since the set of
  $0$-simplices of a simplicial set is the splitting of an idempotent
  on the set of $1$-simplices---so that $G^\mathsf{D}$ is also
  conservative. Thus by the Beck monadicity theorem 
  $G^\mathsf{D}$ is monadic, and so the comparison functor
  $[\Delta^\mathrm{op}, \cat{Set}] \rightarrow ([\Delta^\mathrm{op},
  \cat{Set}]^\mathsf{D})^{\mathsf{\tilde D}}$ is an equivalence.
\end{proof}

Combining this with the characterisation of the
essential image of~\eqref{eq:23} yields:
\begin{Cor}
  \label{cor:1}
  The category $\cat{Op\C at}_1$ of unary operadic categories is
  isomorphic to the full (reflective) subcategory of
  $[\Delta^\mathrm{op}, \cat{Set}]$ on those simplicial sets $C$ for
  which $D(C)$ satisfies the Segal condition.
\end{Cor}

Explicitly, the simplicial set $C$ giving the ``undecking'' of a unary
operadic category $\C$ has as $0$-simplices, the chosen terminal
objects of $\C$, and as $(n+1)$-simplices the $n$-simplices of the
nerve of $\C$. The faces of a $1$-simplex $X$ are 
\begin{equation*}
  \varphi(1_X) \xrightarrow{X} uX
\end{equation*}
where we write $\varphi(f)$ for the unique fibre $f^{-1}(\ast)$ of a
map $f \colon Y \rightarrow X$ of $\C$. The faces of a $2$-simplex
$f \in \C(Y,X)$ are given by
\begin{equation*}
  \cd[@!C@C-2.5em@R-0.5em]{
    & \varphi(1_X) \ar[dr]^-{X} \\
    \varphi(1_Y) \ar[ur]^-{\varphi(f)} \ar[rr]_-{Y} &
    \cell[0.4]{u}{f} & uX\rlap{ ;}
  }
\end{equation*}
while the faces of a $3$-simplex $(g,f) \in \C(Z,Y) \times \C(Y,X)$
are given by
\begin{equation*}
  \cd[@!C@C-2.6em]{
    & \varphi(1_Y) \cell[0.6]{dr}{g} \ar[drrr]|-{Y}
    \ar[rr]^-{\varphi(f)} & & \varphi(1_X) \cell[0.29]{dl}{f}
    \ar[dr]^-{X} \\
    \varphi(1_Z) \ar[ur]^-{\varphi(g)} \ar[rrrr]_-{Z} & & & &
    uX
  } \ \ \ 
  \cd[@!C@C-2.6em]{
    & \varphi(1_Y) \ar[rr]^-{\varphi(f)}
    \cell[0.3]{dr}{g^f_\ast} & & \varphi(1_X) \cell[0.55]{d}{fg}
    \ar[dr]^-{X} \\
    \varphi(1_Z) \ar[urrr]_-{\varphi(fg)} \ar[ur]^-{\varphi(g)} \ar[rrrr]_-{Z} & & & &
    uX\rlap{ .}
  }
\end{equation*}
The degeneracies are easily written down, and the remaining data is
determined by coskeletality. Note that $D(C)$ is the nerve of $\C$,
which satisfies the Segal condition. Conversely, if $C$ is a
simplicial set for which $D(C)$ satisfies the Segal condition, then
$D(C) \cong \mathrm{N}(\C)$ for a category $\C$, and by working backwards
through the above description we may read off the operadic structure
on $\C$.

\begin{Rk}
  \label{rk:4}
  The condition on a simplicial set $X$ that $D(X)$ should satisfy the
  Segal condition gives half of the axioms for a discrete
  \emph{decomposition space}~\cite{Galvez-Carrillo2018Decomposition}.
  (Decomposition spaces are also known as \emph{$2$-Segal
    spaces}~\cite{Dyckerhoff2012Higher}.) In particular, for any
  discrete decomposition space
  $X \colon \Delta^\mathrm{op} \rightarrow \cat{Set}$, its d\'ecalage
  is a unary operadic category, generalising Example~\ref{ex:2}. For
  example, there is a discrete decomposition space $X$ of
  (combinatorialists') graphs, wherein $X_n$ is the set of graphs with
  a map from the set of vertices to $\underline n$. The corresponding
  unary operadic category has graphs as objects; a map is the opposite
  of a full inclusion of graphs, and the fibre of such a map is the
  induced graph on the complementary set of vertices.

  In fact, the remaining axioms for a discrete decomposition space $X$
  can be expressed in terms of the associated unary operadic category
  $\C$: they say precisely that the fibre functor
  $\varphi \colon D(\C) \rightarrow \C$ is a discrete opfibration.
  This establishes a link with Lack's~\cite{Lack2018Operadic}, which
  characterises operadic categories with object set $O$ in terms of
  certain left-normal \emph{skew
    monoidal}~\cite{Szlachanyi2012Skew-monoidal} structures on
  $\cat{Set}/O$, and provides conditions for these skew structures to
  be genuinely monoidal; in the unary case, the necessary condition
  is, again, that $\varphi$ be a discrete opfibration.
  
  In the following result, the equivalence between (i) and (ii) is
  thus due to Lack; we omit the proof, since the result is not needed
  elsewhere in this paper.
  \begin{Thm*}
    Let $\C$ be a unary operadic category. The following are
    equivalent:
    \begin{enumerate}[(i)]
    \item The fibre functor $\varphi \colon D(\C) \rightarrow \C$ is a
      discrete opfibration;
    \item The associated skew monoidal structure on
      $\cat{Set}/{\mathrm{ob}\,\C}$ is genuinely monoidal;
    \item The ``undecking'' $C$ is a discrete decomposition space.
    \end{enumerate}
  \end{Thm*}

  In fact (cf.~\cite[Remark~7.2]{Lack2018Operadic}) the left-normal
  skew monoidal structures induced by unary operadic categories are
  precisely those whose tensor preserves colimits in each variable;
  these can be identified with \emph{skew monoidales} in the monoidal
  bicategory~$\cat{Span}$, and in this case Lack's characterisation
  reduces to one given by
  Andrianopoulos~\cite{Andrianopoulos2016Skew}. Under this
  identification, the unary operadic categories satisfying the
  equivalent conditions of the above theorem correspond to genuine
  monoidales in $\cat{Span}$: in the language
  of~\cite{Galvez-Carrillo2018Decomposition}, this monoidale is the
  \emph{incidence algebra} of the corresponding discrete decomposition
  space.
\end{Rk}
\begin{Rk}
  \label{rk:1}
  The equivalence of Corollary~\ref{cor:1} is also interesting in the
  other direction. If $\C$ is a unary operadic category derived from a
  category with a zero object and kernels, as in Example~\ref{ex:4},
  then the associated simplicial set is a discrete version of
  Waldhausen's $S_\bullet$ construction.  
\end{Rk}

\section{Modified d\'ecalage}
\label{sec:modified-decalage}

We now wish to expand on Theorem~\ref{thm:1} to give a
characterisation of general operadic categories in terms of
d\'ecalage. As explained in the introduction, the key to this will be
a comonad $\mathsf{D}_m$ on the arrow category $\cat{Cat}^\atwo$ given
on objects by
\begin{equation}\label{eq:5}
 \E \xrightarrow{P} \C \qquad \mapsto \qquad \textstyle\sum_{Y \in 
 \E} \E / Y \xrightarrow{\Sigma_{Y \in \E} P/Y} \sum_{Y \in \E} \C / PY\rlap{ ,}
\end{equation}
which we call \emph{modified d\'ecalage}. In this section, we describe
this comonad, and show that it restricts back to the the lax slice
category $\cat{Cat} /\!/ \S$, identified with the full subcategory of
$\cat{Cat}^\atwo$ on the discrete opfibrations with finite fibres.

While we could describe the comonad $\mathsf{D}_m$ and its coalgebras
by hand, we prefer in the spirit of the rest of the paper to obtain it
by way of more general considerations. The key is the following
construction on a functor $P \colon \E \rightarrow \C$. It begins by
decomposing $\E$ and $\C$ into their connected components:
\begin{equation*}
 \E = \textstyle \sum_{y \in Y} \E_{y} \qquad \quad \text{and} \quad \qquad \C =
 \sum_{x \in X} \C_x\rlap{ .}
\end{equation*}
Now for each $y \in Y$, the restriction of $P$ to $\E_{y}$ must factor
through a single connected component $\C_{fy}$ of
$\C$. If we write
$P_{y} \colon \E_{y} \rightarrow \C_{fy}$ for this
factorisation, then summing the $P_{y}$'s over all $y \in Y$ yields the
first map $L_P$ in a factorisation:
\begin{equation}\label{eq:6}
 \cd[@!C]{
 \textstyle \sum_{y \in Y} \E_{y} \ar[rr]^P \ar[rd]_{L_P} && \sum_{x 
 \in X} \C_{x} \\
 &  \textstyle \sum_{y \in Y} \C_{fy} \ar[ru]_{R_P}
 }
\end{equation}
whose second map $R_P$ maps the $y$-summand to the $fy$-summand via
$1_{\C_{\smash{fy}}}$. Let us call a functor \emph{$\pi_0$-bijective} if,
like $L_P$, the induced function on connected components is
invertible, and \emph{$\pi_0$-cartesian} if, like $R_P$, it maps each
connected component of its domain bijectively onto a connected
component of its codomain. As these two classes of functors are easily
seen to be orthogonal, we have a factorisation system
($\pi_0$-bijective, $\pi_0$-cartesian) on $\cat{Cat}$; and so 
by~\cite[Theorem~5.10]{Im1986On-classes} we have:

\begin{Lemma}
 \label{lem:1}
 The full subcategory $\pi_0\text-\cat{Bij}$ of $\cat{Cat}^\atwo$
 whose objects are the the $\pi_0$-bijective functors is a
 coreflective subcategory. The counit of the coreflection at $P$ is
 given by the morphism $(1, R_P) \colon L_P \rightarrow P$ in
 $\cat{Cat}^\atwo$.
\end{Lemma}

\begin{Rk}
  \label{rk:2}
  Whenever $H \colon \T \rightarrow \B$ is a Grothendieck fibration, 
  there is a factorisation system on $\T$ whose left and right classes
  are, respectively, the maps inverted by $H$, and the cartesian maps
  with respect to $H$. The above factorisation system arises in this
  way from the connected components functor
  $\pi_0 \colon \cat{Cat} \rightarrow \cat{Set}$.
\end{Rk}
Now, if the $P\colon \E\to\C$ of~\eqref{eq:6} is a strictly
local-terminal-preserving functor between categories endowed with
local terminal objects, then there is a \emph{unique} way of endowing
the interposing $\sum_{y} \C_{fy}$ with local terminal
objects such that both $L_P$ and $R_P$ preserve them strictly. It
follows that the ($\pi_0$-bijective, $\pi_0$-cartesian) factorisation
system on $\cat{Cat}$ lifts to $\cat{Cat}_{\ell t}$, and so again
by~\cite[Theorem~5.10]{Im1986On-classes}:

\begin{Lemma}
 \label{lem:2}
 The full subcategory $\pi_0\text-\cat{Bij}_{\ell t}$ of
 $(\cat{Cat}_{\ell t})^\atwo$ whose objects are the $\pi_0$-bijective
 functors is a coreflective subcategory. The counit of the
 coreflection at $P$ is given by the morphism
 $(1, R_P) \colon L_P \rightarrow P$ in $(\cat{Cat}_{\ell t})^\atwo$.
\end{Lemma}

\begin{Rk}
  \label{rk:3}
  The lifting of the $(\pi_0$-bijective, $\pi_0$-cartesian)
  factorisation system from $\cat{Cat}$ to $\cat{Cat}_{\ell t}$ is in
  fact also the lifting of the \emph{comprehensive} factorisation
  system~\cite{Street1973The-comprehensive}, whose classes are the
  final functors and the discrete fibrations. So the
  category $\pi_0\text-\cat{Bij}_{\ell t}$ is equally the full
  subcategory of $(\cat{Cat}_{\ell t})^\atwo$ on the final functors.
\end{Rk}

Now, if we let $\mathsf{L}$ and $\mathsf{L}_{\ell t}$ denote the
idempotent comonads on $\cat{Cat}^\atwo$ and
$(\cat{Cat}_{\ell t})^\atwo$ corresponding to the coreflective
subcategories of the last two lemmas, then it is evident from their
explicit descriptions that $\mathsf{L}_{\ell t}$ is a
\emph{lifting}---in the sense of~\cite{Beck1969Distributive}---of
$\mathsf{L}$ along the strictly comonadic
$(\cat{Cat}_{\ell t})^\atwo \rightarrow \cat{Cat}^\atwo$. It follows
by the proposition in \sec 2 of \emph{ibid}.\ that the composite
adjunction
\begin{equation}\label{eq:8}
 \cd{
 {\pi_0\text-\cat{Bij}_{\ell t}} \ar@<-4.5pt>[r]_-{} \ar@{<-}@<4.5pt>[r]^-{} \ar@{}[r]|-{\top} &
 {(\cat{Cat}_{\ell t})^{\atwo}} \ar@<-4.5pt>[r]_-{} \ar@{<-}@<4.5pt>[r]^-{} \ar@{}[r]|-{\top} &
 {\cat{Cat}^\atwo}
 }
\end{equation}
is also strictly comonadic. Thus, if we define the \emph{modified
 d\'ecalage comonad} $\mathsf{D}_m$ to be the comonad generated by
this adjunction, then we have:

\begin{Prop}
 \label{prop:2}
 The category $(\cat{Cat}^\atwo)^{\mathsf{D}_m}$ of
 $\mathsf{D}_m$-coalgebras is isomorphic over $\cat{Cat}^\atwo$ to
 the full subcategory $\pi_0\text-\cat{Bij}_{\ell t}$ of
 $(\cat{Cat}_{\ell t})^\atwo$ on the $\pi_0$-bijective functors.
\end{Prop}

By combining Proposition~\ref{prop:D-CoAlg} and Lemma~\ref{lem:2}, we
see that the cofree functor
$\cat{Cat}^\atwo \rightarrow (\cat{Cat}^\atwo)^{\mathsf{D}_m}$ sends
the object $P \colon \E \rightarrow \C$ of
$\cat{Cat}^\atwo$ to the object
\begin{equation}\label{eq:7}
 D_m(P) \quad = \quad \textstyle \sum_{Y \in \E} \E / Y \xrightarrow{\Sigma_{Y \in \E} P/Y} \sum_{Y \in \E} \C / PY
\end{equation}
endowed in domain and codomain with the respective local terminals
$1_Y$ and $1_{PY}$ for each $Y \in \E$. Furthermore, the counit at $P$
of the adjunction~\eqref{eq:8} is the map $D_m(P) \rightarrow P$ of
$\cat{Cat}^\atwo$ whose two components
$\sum_{Y} \E / Y \rightarrow \E$ and
$\sum_{Y} \C / PY \rightarrow \C$ are
given by the appropriate copairings of slice projections.

We now show that the comonad $\mathsf{D}_m$ on $\cat{Cat}^\atwo$
restricts to the \emph{lax slice category} $\cat{Cat} /\!/ \S$. The
objects of this category are pairs of a small category $\C$ and a
functor $\abs{\thg}_\C \colon \C \rightarrow \S$, while morphisms
$(\C, \abs{\thg}_\C) \rightarrow (\someothercat,
\abs{\thg}_\someothercat)$ are pairs of a functor $F$ and natural
transformation $\nu$ fitting into a diagram:
\begin{equation}\label{eq:12}
 \cd[@!C]{
 {\C} \ar[rr]^-{F} \ar[dr]_-{\abs{\thg}_\C} & \rtwocell{d}{\nu} &
 {\someothercat}\rlap{ .} \ar[dl]^-{\abs{\thg}_\someothercat} \\ &
 {\S}
 }
\end{equation}

To embed $\cat{Cat} /\!/\S$ into $\cat{Cat}^\atwo$, we use the
\emph{category of elements} construction. For a functor
$\somesetvaluedfunctor \colon \C \rightarrow \cat{Set}$, its category
of elements $\mathrm{el}(\somesetvaluedfunctor)$ has objects given by
pairs $(X \in \C, i \in \somesetvaluedfunctor X)$, and maps
$(Y,j) \rightarrow (X,i)$ given by maps $f \in \C(Y,X)$ with
$(\somesetvaluedfunctor f)(j) = i$. Associated to the category of
elements we have a discrete opfibration
$\pi_\somesetvaluedfunctor \colon \mathrm{el}(\somesetvaluedfunctor )
\rightarrow \C$ sending $(X,i)$ to $X$; recall that a functor
$P \colon \E \rightarrow \C$ is a \emph{discrete opfibration} if, for
every $Y \in \E$ and $f \colon P Y \rightarrow X$ in $\C$, there is a
unique map $\bar f \colon Y \rightarrow \bar X$ with $P \bar f = f$.
In particular, to each $(\C, \abs{\thg}_\C) \in \cat{Cat} /\!/\S$ we
can associate the discrete opfibration
$P_\C \colon \E_\C \rightarrow \C$ obtained as the projection from the
category of elements of
$\abs{\thg}_\C \colon \C \rightarrow \S \hookrightarrow \cat{Set}$.

\begin{Prop}
 \label{prop:3}
 The assignation
 $(\C, \abs{\thg}_\C) \mapsto (P_\C \colon \E_\C \rightarrow \C)$ is the
 action on objects of a fully faithful functor
 $\Upsilon \colon \cat{Cat} /\!/ \S \rightarrow \cat{Cat}^\atwo$. Its
 essential image comprises the discrete opfibrations with finite
 fibres, and choosing an isomorphism with an object in the image
 amounts to endowing each of these fibres with a linear order.
\end{Prop}

While this result is well known, we prove it for the sake of
self-containedness.

\begin{proof}
 If $(\C, \abs{\thg}_\C)$ and
 $(\someothercat, \abs{\thg}_\someothercat)$ are objects of
 $\cat{Cat} /\!/\S$, then a map $P_\C \rightarrow P_\someothercat$ of
 $\cat{Cat}^\atwo$ is a commutative square
 \begin{equation*}
 \cd{
 \E_\C \ar[r]^{G}\ar[d]_{P_\C} & \E_\someothercat
 \ar[d]^{P_\someothercat} \\
 \C \ar[r]^F & \someothercat\rlap{ .}
 }
 \end{equation*}
 Commutativity forces
 $G(X,i) = (FX, \nu_X(i))$ for suitable
 $\nu_X(i) \in \abs{FX}_\someothercat$, so yielding functions
 $\nu_X \colon \abs{X}_\C \rightarrow \abs{FX}_\someothercat$, which by
 applying $G$ to morphisms we see are natural in $X$. So every
 map $P_\C \rightarrow P_\someothercat$ arises from a lax
 triangle~\eqref{eq:12}, and it is easy to see that any such triangle
 induces a map $P_\C \rightarrow P_\someothercat$ in this manner.

 So $\Upsilon$ is well defined and fully faithful. As for its
 essential image, it is well known (and easily proved) that
 $H \colon \E \rightarrow \C$ is a discrete opfibration just when it
 is isomorphic over $\C$ to
 $\pi_\somesetvaluedfunctor \colon \el(\somesetvaluedfunctor )
 \rightarrow \C$ for some functor
 $\somesetvaluedfunctor \colon \C \rightarrow \cat{Set}$. In this
 case, $H$ will have finite fibres just when
 $\pi_\somesetvaluedfunctor $ does so, which happens just when each
 $\somesetvaluedfunctor (B)$ is finite. But such a
 $\somesetvaluedfunctor $ may always be replaced by an isomorphic one
 which factors through $\S \subseteq \Set$, and so the discrete
 opfibration $H$ has finite fibres just when it is in the essential
 image of $\Upsilon$.

 Finally, the fibre of $P_\C \colon \E_\C \rightarrow \C$ over 
 $X \in \C$ is the set $\{(X,i) : i \in \abs{X}_\C\}$ which inherits a
 linear order from $\abs{X}_\C$. So any specified isomorphism $H \cong
 P_\C$ induces by transport of structure a linear order on each fibre
 of $H$. Conversely, given a linear order on the fibres of $H$, we
 may reconstruct an isomorphism with $P_\C$ by requiring each map on
 fibres to be a \emph{monotone} isomorphism.
\end{proof}

We now show that the modified d\'ecalage comonad $\mathsf{D}_m$ on
$\cat{Cat}^\atwo$ restricts back to a comonad on $\cat{Cat} /\!/ \S$.

\begin{Prop}
 \label{prop:4}
 The essential image of
 $\Upsilon \colon \cat{Cat} /\!/ \S \rightarrow \cat{Cat}^\atwo$ is
 closed under the action of modified d\'ecalage, which thus restricts
 to a comonad $\mathsf{D}_m$ on $\cat{Cat} /\!/ \S$. The category of
 coalgebras $(\cat{Cat} /\!/ \S)^{\mathsf{D}_m}$ is isomorphic to the
 lax slice $\cat{Cat}_{\ell t} /\!/ \S$.
\end{Prop}

\begin{proof}
 Given $(\C, \abs{\thg})$ in $\cat{Cat} /\!/\S$, applying $D_m$ to
 the corresponding $P_\C \colon \E_\C \rightarrow \C$ in
 $\cat{Cat}^\atwo$ yields by~\eqref{eq:7} the functor
 \begin{equation}\label{eq:13}
 \textstyle \sum_{X \in \C, i \in {\abs X}} \E_\C / (X,i) \xrightarrow{\Sigma_{X, i
 } P_\C/(X,i)} \sum_{X \in \C, i \in {\abs X}} \C / X\rlap{ .}
 \end{equation}
 We must show this is a discrete opfibration with finite fibres. Since
 functors of this kind are closed under coproducts, it suffices to
 show that each $P_\C / (X,i) \colon \E_\C / (X,i) \rightarrow \C / X$
 is a discrete opfibration with finite fibres. It is a discrete
 opfibration since it is a slice of the discrete opfibration $P_\C$;
 as for the fibres, given $f \colon Y \rightarrow X$ in $\C / X$, the
 objects over it in $\E_\C / (X,i)$ are maps of $\E_\C$ of the form
 $f \colon (Y,j) \rightarrow (X,i)$, which are indexed by the
 \emph{finite} set $\{\,j \in \abs{Y} : \abs{f}(j) = i\,\}$.

 It follows that $\mathsf{D}_m$ restricts back to a comonad on
 $\cat{Cat} /\!/\S$, and the corresponding category of coalgebras fits
 into a pullback
 \begin{equation*}
 \cd{
 {(\cat{Cat}/\!/\S)^{\mathsf{D}_m}} \pullbackcorner \ar[r]^-{} \ar[d]_{U} &
 {(\cat{Cat}^\atwo)^{\mathsf{D}_m}} \ar[d]^{U} \\
 {\cat{Cat}/\!/\S} \ar[r]^-{\Upsilon} &
 {\cat{Cat}^\atwo}\rlap{ .}
 }
 \end{equation*}

 Now given $(\C, \abs{\thg}) \in \cat{Cat} /\!/ \S$, endowing its
 image $P_\C \colon \E_\C \rightarrow \C$ under $\Upsilon$ with
 $\mathsf{D}_m$-coalgebra structure means, first of all, endowing $\C$
 with local terminal objects. Having done this, we must endow $\E_\C$
 with local terminals such that $P_\C$ preserves them, and it is easy
 to see that the unique way of doing this is by choosing the set
 $\{(X, i) : \text{$X$ is local terminal in $\C$, $i \in \abs X$}\}$.
 Finally, to assert that $P_\C$ is $\pi_0$-bijective, there must be a
 \emph{unique} $(X,i)$ over each chosen local terminal of $\C$, which
 is to say that $\abs{X} = \underline 1$ for each local terminal of
 $\C$. So objects of $(\cat{Cat} /\!/ \S)^{\mathsf{D}_m}$ are in
 bijection with those of $\cat{Cat}_{\ell t} /\!/ \S$. The argument on
 maps is similar and left to the reader.
\end{proof}

\section{Characterising lax-operadic categories}
\label{sec:char-lax-oper}

In this section, we take the procedure employed in
Section~\ref{sec:line-oper-categ} for the d\'ecalage comonad on
$\cat{Cat}$---considering its category of coalgebras, then the monad
induced on the category of coalgebras, and then the algebras for
\emph{that} monad---and apply it to the modified d\'ecalage comonad on
$\cat{Cat} /\!/ \S$. By doing so, we come very close to obtaining a
characterisation of operadic categories. What we in fact characterise
are instances of the more general notion of \emph{lax-operadic}
category. These generalise operadic categories by replacing the fact
of the commutativity of the triangles~\eqref{axiomstriangle} by the
data of coherent $2$-cells filling these triangles. 

\begin{Defn}
  \label{def:lax-operadic-category}
  A \emph{lax-operadic category} is given by the following data, which
  augment those of an operadic category by the addition
  of~\ref{data:lax-operadic-Q4}:
  \begin{enumerate}[label=(D\tstyle{\arabic*})]
  \item  A category $\C$ endowed with local
    terminal objects;
    \item A cardinality functor
      $\abs{\thg} \colon \C \to \S$;
    \item  For all $X\in \C$ and 
      $i \in \abs{X}$ a fibre functor
      $\fibre_{X,i} \colon \C/X \to \C$ notated as before;
  \item \label{data:lax-operadic-Q4} For each $f \colon Y \to X$ in
    $\C$ and $i \in \abs X$, a \emph{relabelling function}
    \begin{equation*}
      \gamma_{f,i} \colon {\abs f}^{-1}(i) \rightarrow
      \abs{\smash{f^{-1}(i)}}\rlap{ .}
    \end{equation*}
  \end{enumerate}
  These data are subject to the following axioms, which are as for an
  operadic category, except that~\ref{axQ:BM-67} and
  and~\ref{axQ:BM-fibres-of-local-fibres} are suitably modified to
  take account of the relabelling functions
  of~\ref{data:lax-operadic-Q4}. In
  stating~\ref{axQ:BM-fibres-of-local-fibres-lax}, we write $\gamma j$
  and $\varepsilon j$ for the images of $j \in \abs{f}^{-1}(i)$ under
  $\gamma_{f,i} \colon \abs{f}^{-1}(i) \rightarrow \abs{f^{-1}(i)}$
  and
  $\varepsilon_{\abs f, i} \colon \abs{f}^{-1}(i) \rightarrow \abs Y$.
  
  \begin{enumerate}[label=(A\tstyle{\arabic*}-lax)]
  \item[(A\tstyle 1)\phantom{-lax}] If $X$ is a local terminal then
    $\abs{X}=\underline 1$;
  \item[(A\tstyle 2)\phantom{-lax}]  For all $X \in \C$ and
    $i \in \abs X$, the object $(1_X)^{-1}(i)$ is chosen terminal;
  \addtocounter{enumi}{2}
  \item For all $g \colon fg \rightarrow f$ in $\C / X$ and
    $i \in \abs X$, the fibre map is compatible with relabelling, in the
    sense that
    $\abs{\smash{g^f_i}} \circ \gamma_{fg,i} = \gamma_{f,i} \circ
    \abs{g}^{\abs f}_i$;
    \label{axQ:BM-67-lax}
  \item[(A\tstyle 4)\phantom{-lax}] For $X \in \C$, one has
    $\tau_X^{-1}(\ast) = X$, and for $f \colon Y \to X$, one has
    $f^{\tau_X}_\ast = f$;
  \addtocounter{enumi}{1}
  \item \label{axQ:BM-fibres-of-local-fibres-lax}
    For $g \colon fg \rightarrow f$ in $\C/X$, $i \in \abs X$ and
    $j \in \abs{f}^{-1}(i)$ one has that
    $(g^f_i)^{-1}(\gamma j) = g^{-1}(\varepsilon j)$ and that the
    square left below commutes:
    \begin{equation}\label{eq:16}
      \cd{
        (\abs{g}^{\abs f}_i)^{-1}(j) \ar[r]^-{{\bar
            \gamma}_{fg,i}} \ar@{=}[d]_-{} &
        \abs{\smash{g^f_i}}^{-1}(\gamma j) \ar[d]^-{\gamma_{g^f_i,
            \gamma j}} \\
        \abs{g}^{-1}(\varepsilon j) \ar[r]_-{\gamma_{g, \varepsilon j}} &
        \abs{\smash{g^{-1}(\varepsilon j)}}
      } \quad
      \cd{
        (\abs{g}^{\abs f}_i)^{-1}(j) \ar[r]^-{{\bar
            \gamma}_{fg,i}} \ar[d]_-{\varepsilon_{\abs{g}^\abs{f}_i, j}} &
        \abs{\smash{g^f_i}}^{-1}(\gamma j)
        \ar[d]^-{\varepsilon_{\abs{\smash{g^f_i}}, \gamma j}} \\
        {\abs{fg}^{-1}(i)} \ar[r]_-{\gamma_{fg, i}} &
        \abs{\smash{(fg)^{-1}(i)}}
      }
    \end{equation}
    where $\bar \gamma_{fg,i}$ is the unique map making the square right
    above commute. Given moreover $h \colon fgh \rightarrow fg$ in
    $\C / X$, one has
    $\smash{(h^{fg}_i)^{g^f_i}_{\gamma j} = h^g_{\varepsilon j}}$.

  \end{enumerate}

  A strictly local-terminal-preserving functor
  $F \colon \C \rightarrow \someothercat$ between lax-operadic
  categories is called a \emph{lax-operadic functor} if it comes endowed
  with a natural family of relabelling functions
  $\nu_{X} \colon \abs X \rightarrow \abs{FX}$, which are compatible
  with fibre functors in the sense of rendering commutative each diagram
  of the form:
  \begin{equation*}
    \xymatrix{
    \C/X \ar[rr]^-{\fibre_{X,i}} \ar[d]_{F/X} && \C \ar[d]^F \\
    \someothercat/FX \ar[rr]_-{\fibre_{FX,\nu_X(i)}} &&
    \someothercat\rlap{ ;} }
  \end{equation*}
  in other words, we have $F(f^{-1}(i)) = (Ff)^{-1}(\nu_X(i))$ and
  $F(g^f_i) = (Fg)^{Ff}_{\nu_X(i)}$ for all
  $g \colon fg \rightarrow f$ in $\C / X$ and $i \in \abs X$. We write
  $\cat{Lax \O p\C at}$ for the category of lax-operadic categories
  and lax-operadic functors.
\end{Defn}

It is perhaps worth type-checking the display
in~\ref{axQ:BM-fibres-of-local-fibres-lax} to see that it makes sense.
In the left square, the left edge is well-defined simply by computing
cardinalities of fibres; while the right edge is well-defined by the
equality $(g^f_i)^{-1}(\gamma j) = g^{-1}(\varepsilon j)$ asserted
directly beforehand. In the right square, for the factorisation
$\bar \gamma_{fg,i}$ to exist, we must know that $\gamma_{fg,i}$ maps
each $k \in \abs{fg}^{-1}(i)$ with $\abs{g}^{\abs f}_{i}(k) = j$ to an
element $k' \in \abs{(fg)^{-1}(i)}$ with
$\abs{\smash{g^f_i}}(k')= \gamma_{f,i}(j)$; but this follows from the
equality
$\abs{\smash{g^f_i}} \circ \gamma_{fg,i} = \gamma_{f,i} \circ
\abs{g}^{\abs f}_i$ asserted in~\ref{axQ:BM-67-lax}. 




We now begin our abstract rederivation of lax-operadic categories in
terms of modified d\'ecalage. 
Recall that in Proposition~\ref{prop:4}, we exhibited the category of
coalgebras for the modified d\'ecalage comonad on $\cat{Cat} /\!/\S$
as isomorphic to the lax slice category $\cat{Cat}_{\ell t}  /\!/\S$.
Thus we are justified in giving:

\begin{Defn}
 \label{def:3}
 The \emph{modified d\'ecalage monad} $\mathsf{\tilde D}_m $ on
 $\cat{Cat}_{\ell t} /\!/ \S \cong (\cat{Cat} /\!/ \S)^{\mathsf{D}_m}$ is the monad induced by the forgetful--cofree adjunction
  $(\cat{Cat} /\!/ \S)^{\mathsf{D}_m}
  \leftrightarrows \cat{Cat} /\!/ \S$.
\end{Defn}

Towards a concrete description of the modified d\'ecalage
monad, we note that the sets $\{ j \in \abs{Y} : \abs{f}(j) = i\}$
giving the fibres of~\eqref{eq:13} inherit linear orders from
$\abs Y$, so that we may use the last clause of
Proposition~\ref{prop:3} to obtain a particular instantiation of the
forgetful--cofree adjunction for the modified d\'ecalage comonad on
$\cat{Cat} /\!/\S$. The cofree functor
$\cat{Cat} /\!/\S \rightarrow (\cat{Cat} /\!/\S)^{\mathsf{D}_m}$ sends
an object $(\C, \abs{\thg})$ to the object
$(D_m(\C), \abs{\thg}_{D_m(\C)})$, where
$D_m(\C) = \Sigma_{X \in \C, i \in \abs X} \C / X$ is the codomain
of~\eqref{eq:13}, with the chosen terminal $1_X$ in the connected
component indexed by $(X,i)$, and where
$\abs{\thg}_{D_m(\C)} \colon D_m(\C) \rightarrow \S$ is defined on
objects and morphisms by
\begin{equation}
 \label{eq:14}
 \begin{aligned}
 (X \in \C, i \in \abs X, f \colon Y \rightarrow X) \qquad &\mapsto
 \qquad \abs{f}^{-1}(i) \\
 (X,i,fg) \xrightarrow{g} (X,i,f) \qquad & \mapsto \qquad
 \abs{fg}^{-1}(i) \xrightarrow{\smash{\abs{g}^{\abs f}_i}}
 \abs{f}^{-1}(i)\rlap{ .}
 \end{aligned}
\end{equation}

The counit at $(\C, \abs{\thg})$ of the forgetful--cofree adjunction
is given by a lax triangle
\begin{equation}\label{eq:20}
 \cd[@!C@C-1.5em]{
 {D_m(\C)} \ar[rr]^{E_\C} \ar[dr]_-{\abs{\thg}_{D_m(\C)}} &
 \rtwocell{d}{\varepsilon_\C} &
 {\C} \ar[dl]^-{\abs{\thg}} \\ &
 {\S}
 }
\end{equation}
wherein the functor
$E_\C \colon \Sigma_{X \in \C, i \in \abs X} \C / X \rightarrow \C$ is
the copairing of the slice projections, and the natural transformation
$\varepsilon_\C$ has component at $(X,i,f)$ given by the map
$\varepsilon_{\abs f,i} \colon {\abs f}^{-1}(i) \rightarrow \abs Y$
of~\eqref{eq:10}. We now use this to read off a description of the
modified d\'ecalage monad on
$(\cat{Cat} /\!/ \S)^{\mathsf{D}_m} \cong \cat{Cat}_{\ell t} /\!/ \S$.

\begin{enumerate}[(i),itemsep=0.25\baselineskip,widest=iii]
\item The underlying functor $\tilde D_m \colon \cat{Cat}_{\ell t}
  /\!/ \S \rightarrow \cat{Cat}_{\ell t} /\!/ \S$ is given on objects by
 $(\C, \abs{\thg}) \mapsto (D_m(\C), \abs{\thg}_{D_m(\C)})$ as above,
 and on morphisms by:
 \begin{equation*}
 \cd[@!C@C-1.8em]{
 {\C} \ar[rr]^-{F} \ar[dr]_-{\abs{\thg}_\C} & \rtwocell{d}{\nu} &
 {\someothercat} \ar[dl]^-{\abs{\thg}_\someothercat} \ar@{}[drr]|-{\!\!\!\!\!\!\!\!\displaystyle\mapsto}& &
 {D_m(\C)} \ar[rr]^-{D_m(F)} \ar[dr]_-{\abs{\thg}_{D_m(\C)}} &
 \rtwocell{d}{D_m(\nu)} &
 {D_m(\someothercat)\rlap{ .}} \ar[dl]^-{\abs{\thg}_{D_m(\someothercat)}} \\ &
 {\S} & & & &
 {\S}
 }
 \end{equation*}
 Here $D_m(F)$ has action on objects
 $(X, i, f) \mapsto (FX, \nu_X(i), Ff)$ and action on maps
 inherited from $F$; while the component of $D_m(\nu)$ at an
 object $(X,i,f)$ is the unique map rendering commutative the square
 \begin{equation*}
 \cd[@C+3em]{
 {\abs{f}_\C^{-1}(i)} \ar[r]^-{D_m(\nu)_{(X,i,f)}}
 \ar[d]_{\varepsilon_{\abs{f}, i}} &
 {\abs{Ff}_\someothercat^{-1}(\nu_X(i))} \ar[d]^{\varepsilon_{\abs{Ff},
 \nu_X(i)}} \\
 {\abs X}_\C \ar[r]^-{\nu_Y} &
 {\abs {FX}}_\someothercat\rlap{ .}
 }
 \end{equation*}
\item The unit
 $\eta_{\C} \colon (\C, \abs{\thg}) \rightarrow {\tilde D}_m(\C,
 \abs{\thg})$ is a \emph{strictly} commuting triangle, whose upper
 edge is the functor $\C \rightarrow D_m(\C)$ sending $X$ to
 $(uX, 1, \tau_X)$ and sending $f \colon Y \rightarrow X$ to
 $f \colon (uX, 1, \tau_X) \rightarrow (uY, 1, \tau_Y)$.
\item The multiplication
 $\mu_\C \colon {\tilde D}_m{\tilde D}_m(\C, \abs{\thg}) \rightarrow
 {\tilde D}_m(\C, \abs{\thg})$ is also a \emph{strictly} commuting
 triangle, whose upper edge is the functor
 \begin{equation}\label{eq:18}
 \textstyle \sum_{(X,i,f) \in D_m \C, j \in
 \abs{f}^{-1}(i)} D_m \C / (X,i,f) \rightarrow 
 \sum_{X \in \C, i \in \abs X} \C / X
 \end{equation}
 defined as follows. Since a typical map in $D_m \C$ is of the form
 $g \colon (X,i,fg) \rightarrow (X,i,f)$, an object of the domain
 of~\eqref{eq:18} comprises the data of
 \begin{equation}\label{eq:15}
 X \in \C, i \in \abs X, f \colon Y \rightarrow X, j \in
 \abs{f}^{-1}(i), g \colon Z \rightarrow Y
 \end{equation}
 while each morphism is of the form
 $h \colon (X,i,f,j,gh) \rightarrow (X,i,f,j,g)$. In these terms, we
 can define the functor~\eqref{eq:18} on objects and morphisms by
 \begin{equation}
 \label{eq:17}
 \begin{aligned}
 (X,i,f,j,g) \quad &\mapsto
 \quad (Y,\varepsilon j, g) \\
 (X,i,f,j,gh) \xrightarrow{\smash h} (X,i,f,j,g) \quad & \mapsto \quad
 (Y,\varepsilon j, gh) \xrightarrow{\smash h}
 (Y,\varepsilon j, g)\rlap{ ,}
 \end{aligned}
 \end{equation}
 where, like before, we write $\varepsilon j$ for
 $\varepsilon_{\abs f,i}(j)$.
\end{enumerate}

Using this description, we can now give our second main result.

\begin{Thm}
 \label{thm:2}
 The category of algebras for the modified d\'ecalage monad
 $\mathsf{\tilde D}_m$ on
 $\cat{Cat}_{\ell t} /\!/ \S \cong (\cat{Cat}/\!/\S)^{\mathsf{D}_m}$
 is isomorphic to the category $\cat{Lax\O p\C at}$ of lax-operadic
 categories.
\end{Thm}

\begin{proof}
  The data~\ref{data:operadic-Q1}--\ref{data:operadic-Q2} and
  axiom~\ref{axQ:BM-abs(lt)} specify exactly an object
  $(\C, \abs{\thg})$ in $\cat{Cat}_{\ell t} /\!/ \S$. Giving the fibre
  functors~\ref{data:operadic-Q3} is equivalent to giving a single
  functor $\varphi \colon D_m(\C) \rightarrow \C$, and the relabelling
  maps of~\ref{data:lax-operadic-Q4} give the components of a natural
  transformation
 \begin{equation}\label{eq:19}
 \cd[@!C@C-2em]{
 \sh{l}{0.25em}{D_m(\C)} \ar[rr]^{\varphi} \ar[dr]_-{\abs{\thg}_{D_m(\C)}} &
 \rtwocell{d}{\gamma} &
 {\C} \ar[dl]^-{\abs{\thg}} \\ &
 {\S}
 }
 \end{equation}
 whose naturality is then asserted by~\ref{axQ:BM-67-lax}. Since
 axiom~\ref{axQ:BM-fibres-of-identities} asserts that $\varphi$
 in~\eqref{eq:19} is a map of $\cat{Cat}_{\ell t}$, we conclude
 that 
 giving the data for a lax-operadic category plus the first three
 axioms is the same as giving an object $(\C, \abs{\thg})$ of
 $(\cat{Cat} /\!/ \S)^{\mathsf{D}_m}$ endowed with a morphism
 $(\phi, \gamma) \colon \tilde D_m(\C, \abs{\thg}) \rightarrow (\C,
 \abs{\thg})$.

 It is not hard to see that~\ref{axQ:BM-fibres-of-tau-maps} is
 equivalent to $(\varphi, \gamma)$ satisfying the the unit axiom
 $(\phi, \gamma) \circ (\eta_{(\C, \abs{\thg})}, 1) = (1_\C,
 1_{\abs{\thg}})$ for a $\mathsf{\tilde D}_m$-algebra; we claim,
 finally, that~\ref{axQ:BM-fibres-of-local-fibres-lax} asserts the
 multiplication axiom given by the equality of pastings:
 \begin{equation}\label{eq:21}
 \cd[@!C@C-1em@R+0.4em]{
 D_m D_m \C \ar[dr]_-{\abs{\thg}_{D_mD_m(\C)}}
 \ar[r]^-{D_m \phi} & D_m \C \ar[r]^-{\phi}
 \rtwocell[0.37]{dl}{D_m(\gamma)} \rtwocell[0.37]{dr}{\gamma}
 \ar[d]|-{\abs{\thg}_{D_m\C}} & \C
 \ar[dl]^-{\abs{\thg}} \\ &
 \S & {}
 } \quad = \quad
 \cd[@!C@C-1.2em@R+0.4em]{
 D_m D_m \C \ar[dr]_-{\abs{\thg}_{D_mD_m(\C)}}
 \ar[r]^-{\mu_\C} & D_m \C \ar[r]^-{\phi}
 \ar[d]|-{\abs{\thg}_{D_m\C}} \twoeq[0.37]{dl}
 \rtwocell[0.37]{dr}{\gamma} & \C\rlap{ .}
 \ar[dl]^-{\abs{\thg}} \\ &
 \S & {}
 }
 \end{equation}

 Now, the functors across the top of~\eqref{eq:21} act on a typical
 object~\eqref{eq:15} of $D_m D_m \C$ by the respective assignations:
 \begin{equation*}
 \begin{array}{l@{\ }l@{\ }l}
 (X,i,f,j,g) &\ \mapsto\ (f^{-1}(i),
 \gamma j, g^f_i) & \ \mapsto\ 
 (g^f_i)^{-1}(\gamma j)\\
 \llap{and \quad} (X,i,f,j,g) &\ \mapsto\ (Y, \varepsilon j, g) & \
 \mapsto\ 
 g^{-1}(\varepsilon j)\rlap{ ,}
 \end{array}
 \end{equation*}
 whose equality is precisely the first clause
 of~\ref{axQ:BM-fibres-of-local-fibres-lax}. On the other hand, at
 this same object~\eqref{eq:15}, the components of the two composite
 natural transformations in~\eqref{eq:21} are given by the two sides
 of the left square of~\eqref{eq:16}---whose equality is the second
 clause of~\ref{axQ:BM-fibres-of-local-fibres-lax}. Finally, the
 actions on a map $h \colon (X,i,f,j,gh) \rightarrow (X,i,f,j,g)$ of
 $D_m D_m \C$ of the functors across the top of~\eqref{eq:21} are
 given by
 \begin{equation*}
 h \ \mapsto\ {h}^{fg}_i \ \mapsto\ 
 ({h}^{fg}_i)^{g^f_i}_{\gamma j} \qquad \text{and} \qquad 
 h \ \mapsto\ h \
 \mapsto\ 
 h^g_{\varepsilon j}\rlap{ ,}
 \end{equation*}
 whose equality is precisely the final clause
 of~\ref{axQ:BM-fibres-of-local-fibres-lax}. This proves that
 ${\mathsf{\tilde D}}_m$-algebras in
 $(\cat{Cat} /\!/ \S)^{\mathsf{D}_m}$ correspond bijectively with
 lax-operadic categories. A similar argument verifies the same for the
 maps between them, and we leave this to the reader.
\end{proof}

\section{Characterising operadic categories}
\label{sec:char-oper-categ}

There is not much left to do to get from the preceding result to our
main result, characterising genuine operadic categories in terms of
d\'ecalage. If we define a morphism of $\cat{Cat}_{\ell t} /\!/ \S$ as
in~\eqref{eq:12} to be \emph{strict} whenever the natural
transformation $\nu$ therein is an identity, then it is immediate from
the preceding proof that:

\begin{Prop}
 \label{prop:5}
 Under the isomomorphism of Theorem~\ref{thm:2}, a
 $\mathsf{\tilde D}_m$-algebra corresponds to an operadic category
 just when its algebra structure map in $\cat{Cat}_{\ell t} /\!/ \S$
 is strict; while a $\mathsf{\tilde D}_m$-algebra morphism
 corresponds to an operadic functor just when its underlying map in
 $\cat{Cat}_{\ell t} /\!/ \S$ is strict.
\end{Prop}

At this point, it is \emph{not} possible to restrict the modified
d\'ecalage comonad $\mathsf{D}_m$ on $\cat{Cat} /\!/ \S$ back to the
strict slice category $\cat{Cat} / \S$, and obtain operadic categories
as algebras for the induced monad ${\mathsf{\tilde D}}_m$ on
$(\cat{Cat} / \S)^{\mathsf{D}_m}$. The reason for this, as noted in
the introduction, is simply that modified d\'ecalage $\mathsf{D}_m$
does not restrict from $\cat{Cat} /\!/ \S$ to $\cat{Cat} / \S$, since
its counit maps~\eqref{eq:20} are only lax triangles.

However, the modified d\'ecalage monad $\mathsf{\tilde D}_m$ on
$(\cat{Cat} /\!/ \S)^{\mathsf{D}_m} \cong \cat{Cat}_{\ell t} /\!/ \S$
\emph{does} interact well with strictness: inspection of the
description following Definition~\ref{def:3} shows that the functor
${\tilde D}_m$ preserves strictness of triangles, and that each unit
and multiplication component is a strict triangle. We are
thus justified in~giving:
\begin{Defn}
 \label{def:1}
 The \emph{modified d\'ecalage monad} $\mathsf{\tilde D}_m$ on
 $\cat{Cat}_{\ell t} / \S$ is the restriction to $\cat{Cat}_{\ell t}
 / \S$ of the modified
 d\'ecalage monad $\cat{Cat}_{\ell t} /\!/ \S$.
\end{Defn}

And so, from Theorem~\ref{thm:2} and Proposition~\ref{prop:5}, our main
result immediately follows:

\begin{Thm}
 \label{thm:4}
 The category of algebras for the modified d\'ecalage monad
 ${\mathsf{\tilde D}}_m$ on $\cat{Cat}_{\ell t}/\S$ is isomorphic to
 the category $\cat{Op\C at}$ of operadic categories.
\end{Thm}

\bibliographystyle{acm}
\bibliography{bibdata}

\end{document}